\documentclass[11pt,reqno,a4paper]{amsart}
\usepackage[normalem]{ulem}
\usepackage[colorlinks]{hyperref}
\usepackage{amssymb,epsfig,graphics}
\usepackage{amsmath}
\usepackage{enumerate}
\usepackage{comment}
\usepackage{tikz}
\usetikzlibrary{arrows}

\newcommand{\vertiii}[1]{{\left\vert\kern-0.25ex\left\vert\kern-0.25ex\left\vert #1 
    \right\vert\kern-0.25ex\right\vert\kern-0.25ex\right\vert}}

\newtheorem{thm}{Theorem}[section]
\newtheorem{lem}[thm]{Lemma}
\newtheorem{prop}[thm]{Proposition}

\newtheorem{cor}[thm]{Corollary}
\newtheorem{defn}[thm]{Definition}

\newtheorem{rem}[thm]{Remark}

\newcommand{\marginnotemx}[1]
           {\mbox{}\marginpar{\tiny\raggedright\hspace{0pt}{{\bf Maxence}$\blacktriangleright$  {\color{magenta} #1}}}}
           
\newcommand{\marginnotewb}[1]
           {\mbox{}\marginpar{\tiny\raggedright\hspace{0pt}{{\bf Wael}$\blacktriangleright$  {\color{blue} #1}}}}

\def\ZZ{\mathcal{Z}}

\def\D{\Delta}
\def\KK{\mathbb{K}}
\def\NN{\mathbb{N}}

\def\RR{\mathbb{R}}

\let\eps=\varepsilon

\def\B{\mathcal{B}}
\def\D{\mathcal{D}}

\def\ZZ{\mathcal{Z}}
\def\RR{{\mathbb R}}
\def\v{{\mathbf v}}
\def\p{{\mathbf p}}
\def\q{{\mathbf q}}

\def\1{{{\mathit 1} \!\!\>\!\! I} }
\renewcommand{\liminf}{\mathop{{\underline {\hbox{{\rm lim}}}}}}
\renewcommand{\limsup}{\mathop{{\overline {\hbox{{\rm lim}}}}}}

\renewcommand{\ZZ}{\mathbb{Z}}
\newcommand{\N}{\mathbb{N}}

\newcommand{\E}{\mathbb{E}}
\newcommand{\vg}{\textbf v}
\newcommand{\pg}{\textbf{p}}

\newcommand{\vx}{\textbf{x}}
\newcommand{\vq}{\textbf{q}}
\newcommand{\vp}{\textbf{p}}
\newcommand{\LG}{\mathcal L}

\newcommand{\hh}{\textbf{h}}
\newcommand{\w}{\textbf {w}}
\begin{document}
\title{Rare events statistics for $\ZZ^d$ map lattices coupled by collision}
\author{Wael Bahsoun}
\address{Department of Mathematical Sciences, Loughborough University,
Loughborough, Leicestershire, LE11 3TU, UK}
\email{W.Bahsoun@lboro.ac.uk}
\author{Maxence Phalempin}
\address{School of Mathematics and Statistics, University of New South Wales, Sydney NSW 2052 Australia.}
\email{m.phalempin@unsw.edu.au}

\keywords{Transfer operators, rare events, coupled map lattices} 
\thanks{The research of W. Bahsoun was supported by EPSRC grant EP/V053493/1. The work of M. Phalempin was partially supported by ARC Laureat Fellowship FL230100088 and by FABuLOuS project within ``Bando Ricercatori a Firenze". M. Phalempin thanks the Department of Mathematical Sciences at Loughborough University for its hospitality during two visits in the course of this work.\\
\noindent{\bf Conflict of interest statement:} The authors have no conflict of interest to disclose.\\
\noindent{\bf Data availability statement:} No datasets were generated or analysed during the current study.
}

\begin{abstract}
Understanding the statistics of collisions among locally confined gas particles poses a major challenge. In this work we investigate $\ZZ^d$-map lattices coupled by collision with simplified local dynamics that offer significant insights for the above challenging problem. We obtain a first order approximation for the first collision rate at a site $\p^*\in \ZZ^d$ and we prove a distributional convergence for the first collision time to an exponential, with sharp error term. Moreover, we prove that the number of collisions at site $\p^*$ converge in distribution to a compound Poisson distributed random variable. Key to our analysis in this infinite dimensional setting is the use of transfer operators associated with the \emph{decoupled} map lattice at site $\p^*$.
\end{abstract}

\subjclass[2000]{37A25, 37A30, 37D20}
\maketitle
\markboth{W. Bahsoun and M. Phalempin}{Rare events statistics for $\ZZ^d$ map lattices}
\section{Introduction}
The study of rare events in dynamical systems is one of the most active branches of research in ergodic theory. An example of a rare event in dynamical systems appears in the study of `open' systems where orbits infrequently escape from the domain, typically by falling into a small `hole' in the phase space, see \cite{BBF14, DY06, PU17}, the recent article \cite{DT23} and references therein. Another important example of a rare event in dynamical systems appears in the analysis of extreme value statistics, see for instance \cite{LFF16}, the recent work \cite{FFRS20, H24} and references therein. In parallel the study of large, possibly infinite-dimensional, systems is of paramount importance in many fields, including mathematics and physics, see for example \cite{DL11,P20,R12,Y13} and references therein. Coupled map lattices are basic models of large dynamical systems. They were investigated by many authors \cite{CF05}, starting with Kaneko \cite{Ka83}, and were first put in a rigorous setting by Bunimovich and Sinaĭ \cite{BS88}. The models studied in \cite{BS88,CF05,Ka83} cover large coupled systems with weak interaction among its units. However, in several important coupled systems, such as those of locally confined gas particles \cite{GG08}, the interaction among its neighboring units is \emph{rare} but \emph{strong}. Proving statistical properties in models such as that of \cite{GG08} is currently untractable\footnote{From ergodic theory point of view, the most successful technique in treating coupled map lattices with discontinuities and singularities is the transfer operator technique, which we employ in this work. Despite tremendous efforts, see the introduction of \cite{BL21} for a recent account on the current state of the art, such a technique is still very far from being ready to treat directly a model like that of \cite{GG08}.}. In this work we investigate $\ZZ^d$-map lattices coupled by collision with simplified local dynamics, introduced by Keller and Liverani \cite{KL09}, where we prove statistical properties that offer significant insights for models such as that of \cite{GG08}.

The system that we consider has two distinctive features: the rarity of the event that we study appears \emph{naturally} in the model, and the dynamical system that describes the model is \emph{infinite dimensional}. More precisely, the infinite dynamical system is defined by a $\ZZ^d$-coupled map lattice where the site dynamic is a chaotic map of the interval, while the coupling is defined via \emph{rare} interaction between two neighbours. But when two neighbouring sites interact their states change drastically: when the dynamics of two neighbouring sites in the lattice fall simultaneously in their `collision' zones, the states of the neighbouring sites get interchanged. By a loose analogy with classical mechanics, this type of coupling is called `coupling by collision'. Indeed, as we mentioned earlier, the model introduced in \cite{KL09} is reminiscent to collision models studied by mathematicians \cite{BGNST} and physicists \cite{GG08}. 

In this work, in the framework of the infinite dimensional model of \cite{KL09}, we obtain a first order approximation for the \emph{first collision rate} at a site $\p^*\in \ZZ^d$ and we obtain a law for the corresponding first hitting time, with sharp error term. Moreover, we prove that the number of collisions at site $\p^*$ converge in distribution to a compound Poisson distributed random variable. Our techniques are based on analysing the spectrum of rare event transfer operators, which to the best of our knowledge have not been explored before in an \emph{infinite dimensional} setting.  A key idea is to observe that such operators can be associated with the \emph{decoupled}\footnote{See equation \eqref{decoupling} for a precise definition of the decoupled dynamics at site $\p^*$.} map lattice at site $\p^*$. 

For a finite dimensional version of the model of \cite{KL09} a different rare event\footnote{In the work of \cite{BS22}, the authors studied the first collision at any site in a finite dimensional system, and considered the first rate of such a collision divided by the number of sites.} was considered in \cite{BS22}. In the work of \cite{BS22}, the expression of the coupling map did not matter, and basically the rare event in that work can be seen as a rare event of a finite dimensional product map (a rare event for completely decoupled and finite dimensional system). Unlike the work of \cite{BS22} we deal with an \emph{infinite dimensional system}. Moreover, the rare event that we study gets really influenced by the coupling map (equation \eqref{coupling} below) of the system. At the technical level, it is worth noting that when dealing with an infinite dimensional phase space, $\ZZ^d$ in our case, standard spectral techniques are not available and hence delicate verification of the assumptions of the framework of \cite{KL99,KL09',K12} is needed to obtain limit laws for the rare event. In addition, the computation of the so-called `extremal index', which is needed in the derivation of the first hitting time law, is quite involved in our setting.


The paper is organised as follows. In section \ref{sec:set} we introduce the collision model, the rare event and we state our main results, Theorem \ref{thm:main1}, Theorem \ref{thm:main2} and Theorem \ref{thm:main3}. We end section \ref{sec:set} by presenting an example that satisfies the main results where we compute the corresponding probabilistic quantities in a concrete situation. Section \ref{sec:pf1} contains the proof of Theorem \ref{thm:main1}, while sections \ref{sec:pf2} and \ref{sec:proofthm3} contain the proofs of Theorems \ref{thm:main1} and \ref{thm:main3} respectively. Section \ref{sec:LYaux} contains the proof of a uniform Lasota-Yorke inequality for an auxiliary operator, which is used on several occasions for different specific operators in previous sections of the paper.

\section{setting and statement of the main result}\label{sec:set}
Let $I=[0,1]$ and $m$ be Lebesgue measure on it. Let $\tau : I\to I$ be a piecewise $C^2$ and piecewise onto, uniformly expanding map with expanding factor $\alpha>1$; i.e., $\exists$ a finite partition $0=\xi_0<\dots<\xi_M=1$ such that, denoting $I_i:=(\xi_{i-1},\xi_{i})$,
$\tau_{|I_i}$ is monotone and onto, $C^2$, with a $C^2$ extension to $\overline I_i$ and  $|\tau'(x)|\geq \alpha$. It is well known that such maps admit an absolutely continuous invariant measure $\mu_\tau$ with $d\mu_\tau/dm=\rho_\tau$; moreover $\inf_{x\in I}\rho_\tau(x)>0$, the system is mixing \cite{BG} and $\rho_{\tau}\in C^1(I)$, see for instance Lecture 1 of \cite{Li04}\footnote{Although the lecture notes of \cite{Li04} deals with smooth expanding circle maps, the same ideas apply in the case of piecewise smooth onto expanding maps of the interval.}. 


For $\vx:=(x_{\pg})_{\pg \in \ZZ^d} \in I^{\ZZ^d}$ let $T_0: I^{\ZZ^d}\rightarrow I^{\ZZ^d} $ be the product map: 
$$
\left(T_0(\vx)\right)_{\pg}=\tau (x_{\pg}).
$$
Let $V^+:=\{e_i, 1\leq i\leq d\}$ be the standard basis of $\RR^d$ and define $V:=V^+ \cup -V^+$. For $\eps>0$, let $\{A_{\eps,-\v}\}_{\v\in V}\subset I$ be a set of disjoint open intervals, each of length $\eps$. Following \cite{KL09}, we consider the coupling
\begin{equation}\label{coupling}
(\Phi_\eps(\vx))_\p=\begin{cases}
       x_{\p+\v} \quad \text{if } x_\p\in A_{\eps,\v}\text{ and } x_{\p+\v}\in A_{\eps,-\v} \text{ for some } \v\in V\\
       x_\p  \quad\quad\text{otherwise.}
       \end{cases}
\end{equation}
We thus define the coupled dynamics $T_\eps:X\to X$ as the composition 
\begin{equation}\label{eq:fullymap}
T_\eps:=T_0\circ\Phi_\eps.
\end{equation}
\subsection{Banach spaces and notation} Before setting the problem and stating the main results, we define Banach spaces and fix some notation which will be needed in the sequel.\\ 
\noindent$\bullet$ Let $\mathcal D$ be the space of functions on $I^{\mathbb Z^d}$ that depend only on a finite number of coordinates and are locally differentiable. Let 
$$\mathcal D_1 :=\{\varphi\in\mathcal D : |\varphi|_\infty\le 1\}.$$ 
For any Borel complex measure $\mu$ defined on $I^{\mathbb Z^d}$, where $I^{\mathbb Z^d}$ is equipped with the product topology, we define
\begin{align}
|\mu|&=\sup_{\varphi\in \mathcal D_1}\mu(\varphi)\\
\|\mu\|&=\sup_{\p\in \ZZ^d}\sup_{\varphi\in\mathcal D_1}\mu(\partial_{\p}\varphi)    
\end{align}
that is, $\|\cdot\|$ is the bounded variation norm and $|\cdot|$ is the total variation norm. Note that $\mathcal B:=\{\|\mu\|<\infty\}$ is a Banach space as mentioned \cite{KL05}; moreover, elements in $\mathcal B$ consist of measures whose finite dimensional marginals are absolutely continuous with respect to Lebesgue and the corresponding density is a function of bounded variation. Throughout the paper we use the following notation:\\ 

\noindent$\bullet$ For $\Lambda \subset \ZZ^d$, let $\pi_{\Lambda} : I^{\ZZ^d} \rightarrow I^{\Lambda}$ be the canonical projection. When $\Lambda:=\{\q\}$, by a slight abuse of notation, we will write $\pi_\q$ instead of $\pi_{\{\q\}}$. 

\noindent$\bullet$ By $|\Lambda |$ we denote the cardinality of the set $\Lambda$.

\noindent$\bullet$ $m_{\ZZ^d}:=Leb^{\otimes \ZZ^d}$ is the Lebesgue measure on $I^{\ZZ^d}$ whereas $m_\Lambda:=\pi_\Lambda*m_{\ZZ^d}$ is the Lebesgue measure on $I^{\Lambda}$. 

\noindent$\bullet$ $\pi_\Lambda*\nu$ is the image measure of $\nu$ through $\pi_\Lambda$.  When $\nu\in \B$, then $\nu_{\Lambda}$ is the density of $\pi_\Lambda*\nu$ with respect to $m_{\Lambda}$.

\noindent$\bullet$ Finally, we use $|f|_{BV}$ to denote the usual bounded variation norm for $f:I\to\mathbb R$.

\subsection{Rare events at one site}
In this work we study rare events at site $\vp^*\in\mathbb Z^d$. For this purpose we study the asymptotics as the size of the collision intervals involved with site $\vp^*\in\mathbb Z^d$ shrinks to $0$. To do this we use $\delta\in(0,\eps]$ to label such intervals, i.e. $A_{\delta,\v}$ and assume that $m(A_{\delta,\v})=\delta$. All other collision intervals that are not involved with site $\p^*$ will be denoted, as before, by $A_{\eps,\v}$ and their size is independent of $\delta$. Consequently, form now on we start to denote the fully coupled system by $T_{\eps,\delta}$. 
For $\v\in V$, let
$$X^{\vg}_{\delta}(\vp^*):=\{\vx\in I^{\ZZ^d}: x_{\pg^*}\notin A_{\delta,\vg} \text{ or } x_{\pg^*+\vg}\notin A_{\delta,-\vg}\}$$
and define
$$X_{0,\delta}(\vp^*):=\bigcap_{\vg\in V}X^{\vg}_{\delta}(\vp^*),$$
\begin{equation}
H_\delta(\vp^*):=\label{eq:rare-event}
I^{\ZZ^d}\setminus X_{0,\delta}(\vp^*).
\end{equation}
Through the paper, for a set\footnote{When not mentioned otherwise, the complement is taken within $I^{\ZZ^d}$.} $E$ and a subset $A \subset E$ we will use the notation $^cA:= E\backslash A$ to refer to the complement of a set $A$ in $I^{\ZZ^d}$. In that regard, $X_{0,\delta}(\vp^*)=^cH_{\delta}(\p^*)$. Moreover, let
\begin{equation}\label{eq:survival}
    X_{n,\delta}(\vp^*)=\cap_{i=0}^{n-1} T_{\eps,\delta}^{-i}X_{0,\delta}(\vp^*)
\end{equation}
and notice that $X_{n,\delta}(\vp^*)$ is the set of points whose orbits under $T_{\eps,\delta}$ did not see a collision at site $\vp^*$ up to time $n-1$. It is also called the survival set up to time $n-1$.

We endow $I^{\ZZ^d}$ with the product topology and we recall that $m_{\ZZ^d}$ denotes Lebesgue measure on $I^{\ZZ^d}$. 
Suppose that
\begin{equation}\label{eq:collisionrate}
    -\lim_{n \to\infty}\frac{1}{n}\ln m_{\ZZ^d}(X_{n,\delta}(\vp^*))
\end{equation}
exists. Then the quantity in the equation \eqref{eq:collisionrate} measures asymptotically, with respect to $m_{\ZZ^d}$, the fraction of orbits that did not see a first collision at site $\vp^*$. We call such a quantity the \emph{first collision rate} under $T_\eps$ at site $\vp^*$ with respect to the measure $m_{\ZZ^d}$. 
We now introduce a `decoupling' operator at site $\p^*$, which will play a crucial role in our work. Let
\begin{equation}\label{decoupling}
(\Phi_{\eps,\p^*}(\vx))_\q=\begin{cases}
(\Phi_{\eps}(\vx))_\q \quad \text{if } \q-\p^*\notin V\cup\{0\}\\
       (\Phi_{\eps}(\vx))_\q \quad \text{if } \v=\q-\p^* \text{ and } x_{\q}\notin A_{\eps,-\v}\\
       x_\q  \quad\quad\text{otherwise.}
       \end{cases}
\end{equation}
Notice that $\Phi_{\eps,\p^*}$ differs from $\Phi_{\eps}$, which was defined in \eqref{coupling}, only in the coordinate $x_{\p^*}$: under $\Phi_{\eps,\p^*}$,
$x_{\p^*}$ is independent of the other coordinates.
Now define
$$T_{\eps,\vp^*}:=T_0\circ\Phi_{\eps,\p^*}$$
and notice that $X_{n,\delta}(\vp^*)$, the survival set defined in \eqref{eq:survival}, is also given by

\begin{equation}\label{eq:survivaldec}
    X_{n,\delta}(\vp^*)=\cap_{i=0}^{n-1} T_{\eps, \p^*}^{-i}X_{0,\delta}(\vp^*);
\end{equation}
that it is the survival set can be defined via the decoupled dynamics at site $\p^*\in \mathbb Z^d$. This is a very important idea in this work for the following reasons: firstly, the decoupled dynamics $T_{\eps,\vp^*}$ does not depend on $\delta$, consequently if we prove that it admits a physical measure, then the measure will depend on $\eps$ but not on $\delta$ and hence this will make this measure suitable to study asymptotics in $\delta$ for the rare event (see Theorem \ref{thm:main1} below). More importantly, if we prove a spectral gap, on $\mathcal B$, for the transfer operator associated with $T_{\eps,\vp^*}$, all its spectral data are independent of $\delta$. The latter will be an important point to obtain a spectral gap for the corresponding rare event transfer operator, via a perturbation argument. 

We define the transfer operators $\LG_{\eps, \vp^*}$ and $\hat\LG_{\eps,\delta, \vp^*}$ associated with $T_{\eps, \vp^*}$ and the rare event respectively as follows. For $\varphi\in\mathcal D$ and $\mu$ a Borel complex measure
$$\int\varphi d\LG_{\eps, \vp^*}\mu=\int\varphi\circ T_{\eps,\p^*}d\mu$$
and let
\begin{equation}\label{eq:ophole}
    \int \varphi d\hat \LG_{\eps, \delta, \vp^*}(\mu)=\int \varphi \circ T_{\eps,\vp^*} 1_{X_{0,\delta}(\p^*)} d\mu.
\end{equation}
Note that by definition
$$|\LG_{\eps,\delta}\mu|\le |\mu|\quad{ and }\quad |\hat\LG_{\eps,\delta,\p^*}\mu|\le |\mu|.$$
The following theorem is the first main result of the paper:

\begin{thm}\label{thm:main1}
For sufficiently small $\eps>0$ 
\begin{enumerate}
\item $\LG_{\eps, \vp^*}$ and $\hat\LG_{\eps,\delta, \vp^*,}$ admit a spectral gap on $\mathcal B$;

\item the first collision rate with respect to $m_{\ZZ^d}$ exists and satisfies 
$$-\lim_{n \to\infty}\frac{1}{n}\ln m_{\ZZ^d}(X_{n,\delta}(\vp^*))=-\ln\lambda_{\eps,\delta},$$
where $\lambda_{\eps,\delta}\in(0,1)$ is the dominant simple eigenvalue of $\hat\LG_{\eps, \vp^*,\delta}$.
\end{enumerate}
\end{thm}
In the next theorem, we study the asymptotic as $\delta \to 0$, but keeping $\eps>0$ fixed small enough as in Theorem \ref{thm:main1}. Denote by $(a_\v, a_{-\v})$ the point that  $A_{\delta,\v}\times A_{\delta,-\v}$ shrinks to as $\delta \to 0$ and  $S:=\{(a_\v, a_{-\v})\}_{\v\in V^+}$ be the set of all such points. We assume that each $a_\v\in I_i$, for some monotonicity interval $I_i$ of $\tau$ as defined at the beginning of section \ref{sec:set} and  we assume that $a_\v$ is the midpoint of $A_{\delta,\v}$, $\v\in V$. Let
\begin{equation}\label{eq:collisionset}
\mathcal N_{\vp^*}=\{(\q,\v):\, \q\in \vp^*+V, \v=\vp^*-\q\}\cup\{\vp^*\}\times V
\end{equation}
and introduce for $(\q,\v)\in\mathcal N_{\vp^*}$
$$A_{\delta,\v}(\q)=\{\vx\in I^{\ZZ^d}\, \text{ and } x_{\q}\in A_{\delta, \v}\}.$$
Note that $\{A_{\delta,\v}(\q)\}_{(\q,\v)\in\mathcal N_{\vp^*}}$
consists only of collision sets that lead to a collision with the dynamics at $\vp^*$. We introduce the subset $S^{rec}\subset S$ that consists of recurrent points of $S$, under the dynamics $T_{\eps,\delta}$, in the following sense:

\begin{equation}\label{eq:recset}
\begin{split}
    S^{rec}:=&\{(a_\v, a_{-\v}) \in S, \exists k \in \N, \exists \vx\in I^{\ZZ^d}, (x_{\p^*},x_{\p^*+\v})=(a_\v, a_{-\v}) \text{ and }\\
    &\exists (\q,\v')\in \mathcal N_{\p^*}\text{ so that }   (\tau^ka_{\v}, \tau^ka_{-\v})=(a_{\v'}, a_{-\v'})\\ &\text{ and } \Psi_k^{\p^*+\v}(\vx)=\q\},
    \end{split}
\end{equation}
where $\Psi_k^{\vp} : I^{\mathbb Z^d} \to \mathbb Z^d$ is an \textbf{index} map such that $(T_{\eps,\p^*}^k \vx)_{\Psi_k^{\vp}(\vx)}=\tau^k(x_{\p})$. It is defined recursively as follows:
\begin{align*}
\left\{\begin{array}{r l c}
  \Psi_1^{\vp}(\vx)&=&\p+\v \text{ if } x_\p\in A_{\eps,\v}\text{ and } x_{\p+\v}\in A_{\eps,-\v} \\
  &&\text{ for some } \v\in V \text{ whenever } \p^* \notin \{\p,\p+\v\} \\ 
\Psi_1^{\vp}(\vx)&=& \p \text{ otherwise}\\
\Psi_k^{\vp}(\vx)&=&\Psi_{k-1}^{\vp}(\vx)+\Psi_1^{\Psi_{k-1}^{\vp}(\vx)}(T_{\eps,\vp^*}^{k-1}\vx) \text{ if }k\neq 0.
\end{array}
\right.
\end{align*}
Note that the set $S^{rec}$ can strictly contain the set of periodic points of $T_{\eps,\delta}$ in $S$. Indeed, some point $S^{rec}$ can return to $S^{rec}$ without being periodic of $T_{\eps,\delta}$. For each $(a_{\v},a_{-\v})\in S^{rec}$, let $\mathbb K(\v,\v')$ be the set of integers $k$ such that $(\tau^{k+1}a_{\v},\tau^{k+1}a_{-\v})=(a_{\v'},a_{-\v'})$. To each such $k$ we associate the set $J_k$ composed of $j_1\leq \dots \leq j_r\leq k$, the recurrent time into $S^{rec}$ up to time $k$ (i.e , $(\tau^{j_i}(a_{\v}),\tau^{j_i}(a_{-\v}))=(a_{\w_{j_i}},a_{-\w_{j_i}})$).\\

The following notation will appear below in the statement of the second main result of the paper. Using the result of Theorem \ref{thm:main1}, we denote the eigenmeasure corresponding to the eigenvalue 1 for the operator $\LG_{\eps, \vp^*}$ by $\mu_{\eps, \p^*}\in\mathcal B$. Fix the box 
$$\Lambda:=\vp^* + \{-(k+1),\dots,(k+1)\}^d$$ 
around the vector $\vp^*$ and let $\rho_{\eps,\Lambda}$ be the density of $\pi_{\Lambda}*\mu_{\eps,\vp^*}$.\\ 

The \emph{first hitting/collision time} at site $\vp^*$ is defined by
$$t_\delta(\vx)=\inf\{n\ge 0:\,T^n_{\eps,\vp^*}(\vx)\in H_\delta(\vp^*)\}.$$

\begin{thm}\label{thm:main2} Let $\eps>0$ be small enough as in Theorem \ref{thm:main1} and\\ $\delta\in(0,\eps]$ be the size of intervals in $\{A_{\delta,\v}(\q)\}_{(\q,\v)\in\mathcal N_{\vp^*}}$. Let $-\ln\lambda_{\eps,\delta}$ denote the corresponding first collision rate. Then
\begin{enumerate}
\item as $\delta\to 0$
$$\lambda_{\varepsilon,\delta}=1-\mu_{\eps,\vp^*}(H_\delta(\vp^*))\cdot\theta(1+o(1)),$$ 
 with $\theta\in(0,1]$. In particular, if $S^{rec}= \emptyset$ then $\theta=1$; otherwise $\theta\in(0,1)$ and is given by:
 \begin{align*}
\theta=1-\sum_{\v,\v'\in V}\sum_{k\in {\KK(\v,\v')}}q_{k}(\v,\v'),
\end{align*}
where $q_{k}(\v,\v')$ is non zero whenever $(\tau^{k+1}(a_{\v}),\tau^{k+1}(a_{-\v}))=(a_{\v'},a_{-\v'})$ and is defined as follows : 

\begin{align*}
&q_{k}(\v,\v')\nonumber=\frac{1}{\sum_{(q,\v)\in \mathcal N_{\vp^*}}\rho_{\tau}(a_\v)\frac{\rho_{\eps,\q}(a_{-\v}^+)+\rho_{\eps,\q}(a_{-\v}^-)}{2}}\times\nonumber\\
    &\frac{\rho_{\tau}(a_{\v})}{|(\tau^{k+1})'(a_{\v})|}\left(\frac{1}{2|(\tau^{k+1})'(a_{-\v})|}\hat \rho_{\eps,\Lambda,k}(a_{-\v}^+)+\frac{1}{2|(\tau^{k+1})'(a_{-\v})|}\hat \rho_{\eps,\Lambda,k}(a_{-\v}^-)\right),
\end{align*}
with
$$
\hat{\rho}_{\eps,\Lambda,k}=\E(1_{ (\Psi_k^{\q}=\p^*+\v') \cap \bigcap_{i=1}^r \,^c(\Psi_{j_i}^{\q}=\p^*+\w_{j_i})}\rho_{\eps,\Lambda}(.)|I^{\{q\}}).
$$

\item Moreover, $\exists$ $C>0$, $\xi_\delta>0$, with $\lim_{\delta\to 0}\xi_\delta=\theta$, such that for all $t>0$
$$\left|\mu_{\eps,\vp^*}\Big\{t_\delta\ge\frac{t}{\xi_\delta\mu_{\eps,\vp^*}(H_\delta(\vp^*))}\Big\}-e^{-t}\right|\le (t\vee 1)e^{-t} C\delta^2\left|\ln(\delta^2)\right|.$$
\end{enumerate}
\end{thm}
 The key quantity $\theta$ explicitly stated in the above theorem is the \emph{extremal index} associated with the event of having a collision at site $\p^*$. Thus, $S^{rec}\neq \emptyset$ indicates clustering of collisions at site $\p^*$.

Next, we are also interested in studying the distribution of number of collisions at a particular site $\p^*$. For this purpose, we study the following process 
\begin{equation}\label{eq:counting}
Z_\delta(t):=\sum_{k=1}^{\left\lfloor \frac{t}{\mu_{\eps,\p^*}(H_\delta(\p^*))}\right\rfloor}1_{H_\delta(\p^*)}\circ T_{\eps,\delta}^k,
\end{equation}
where, $t>0$, $T_{\eps,\delta}$ is the fully coupled map lattice defined in \eqref{eq:fullymap}, with $\delta\in(0,\eps]$ being the size of collision zones involved with site $\p^*$.

\begin{rem}
Note that as $\delta\to 0$ the fully coupled map lattice $T_{\eps,\delta}$ converges to the decoupled map, $T_{\eps, \p^*}$. Therefore, we consider in \eqref{eq:counting} the scaling $\lfloor \frac{t}{\mu_{\eps,\p^*}(H_\delta(\p^*))}\rfloor$, where $\mu_{\eps,\p^*}$ is the measure associated with $T_{\eps, \p^*}$.
\end{rem} 

\begin{rem}
Generally in classical dynamical systems literature, see for instance \cite{D04, HV09}, when studying the statistics of number of visits to a set the dynamics does not depend on the parameter of the target set. However, since we want to count the collisions of the fully coupled system that take place at site $\p^*$, it is natural in our model to consider the process $Z_\delta(t)$ \eqref{eq:counting} where both the target set and the dynamics depend on the size, $\delta$, of the collision zones associated with $\p^*$. The dependence on $\delta$ in the dynamics $T_{\eps,\delta}$ leads to a `non-standard' characteristic function of the iid sequence associated with the limiting compounded Poisson process in Theorem \ref{thm:main3} below.
\end{rem}

To state our next result, in addition to $S^{rec}$ defined above in \eqref{eq:recset}, we need to introduce the set $\tilde{S}^{rec}$ as follows: 
\begin{equation}\label{eq:recset}
\begin{split}
    \tilde S^{rec}:=&\{(a_\v, a_{-\v}) \in S, \exists k \in \N, \v'\in V, (\tau^k a_{\v},\tau^ka_{-\v})=(a_{\v'},a_{-\v'}) \\
    &\text{ s.t } \exists \vx\in I^{\ZZ^d}, (x_{\p^*},x_{\p^*+\v})=( a_{-\v}, a_\v) \text{ and }\Psi_k^{\p^*+\v}(\vx)=\p^*-\v'\}.\nonumber
    \end{split}
\end{equation}

Note that the main difference between $S^{rec}$ and $\tilde S^{rec}$ is that in $S^{rec}$ we consider $(x_{\p^*},x_{\p^*+\v})=( a_{\v}, a_{-\v})$, while in $\tilde S^{rec}$, we consider $(x_{\p^*},x_{\p^*+\v})=( a_{-\v}, a_\v)$, there is an intervertion coming from the fact we consider recurrence linked to the map $T_{\epsilon,\delta}$ instead of $T_{\epsilon,\p^*}$.

For $(a_{\v},a_{-\v})\in S^{rec}\cup \tilde S^{rec}$, let $\tilde{\mathbb K}(\v,\v')$ be the set of integers $k$ such that $(\tau^{k+1}a_{\v},\tau^{k+1}a_{-\v})=(a_{\v'},a_{-\v'})$ and $J_k$ is the set composed of $j_1\leq \dots \leq j_r\leq k$, the recurrent time into $S^{rec}\cup \tilde S^{rec}$ up to time $k$. Before stating the next theorem, we recall the definition of a compound Poisson random process:

A stochastic process $Z(t) : I^{\mathbb Z^d} \to\mathbb N$ is compound Poisson distributed if there exists a
Poisson process variable $N(t)$ and a sequence of iid random variables
$X_k : I^{\mathbb Z^d}\to \mathbb N$, which is also independent of $N(t)$, such that %
$$Z(t)=\sum_{k=1}^{N(t)}X_i.$$

Recall the process $Z_\delta(t)$ which was defined in \eqref{eq:counting}, which counts the number of collisions at site $\p^*$ from time $k=1$ up to time $\left\lfloor \frac{t}{\mu_{\eps,\p^*}(H_\delta(\p^*))}\right\rfloor$. We prove the following theorem:
\begin{thm}\label{thm:main3}
The process $Z_\delta(t)$ converges in law (for any probability measure $\mu \in \B$; in particular $\mu_{\eps,\p^*}$) to a compound Poisson process $$Z(t)=\sum_{i=1}^{N(t)}X_i,$$ where $(N(t))_t$ is a Poisson process of intensity
$\theta Leb_{|[0,\infty)}$, $\theta\in(0,1]$ is as in Theorem \ref{thm:main2},
and $(X_i)_{i\in \N}$ is an iid sequence whose characteristic function is given by 
$\phi_X(s)=\frac{\tilde \theta(s)(e^{is}-1)}{\theta}+1$, $\tilde\theta:\mathbb R\to\mathbb C$. In particular, $\tilde \theta(s) =1$ when $S^{rec}\cup \tilde S^{rec}=\emptyset$; otherwise,
$$\tilde \theta(s)= 1-\sum_{\v,\v'\in V}\sum_{k\in \tilde{\mathbb{K}}(\v,\v')}\sum_{j=0}^ke^{isj}\left(\beta_k^{(1)}(j,\v,\v')-e^{is}\beta_k^{(2)}(j,\v,\v')\right)$$
where
\begin{equation*}
\begin{split}
&\lim_{\delta \to 0}\frac{\mu_{\eps,\vp^*}\left(H_\delta(\vp^*)\cap T_{\eps,\p^*}^{-1}\left(T_{\eps,\delta}^{-k}H_\delta(\vp^*)\cap \left(\sum_{i=0}^{k-1}1_{H_\delta(\p^*)}\circ T_{\eps,\delta}^i=j\right)\right)\right)}{\mu_{\eps,\vp^*}\left(H_\delta(\vp^*)\right)}\\
&=:\beta_k^{(1)}(j,\v,\v')\\
\end{split}
\end{equation*}
and
\begin{equation*}
\begin{split}
&
\lim_{\delta \to 0}\frac{\mu_{\eps,\vp^*}\left(H_\delta(\vp^*)\cap T_{\eps,\delta}^{-(k+1)}H_\delta(\vp^*)\cap \left(\sum_{i=1}^k1_{H_\delta(\p^*)}\circ T_{\eps,\delta}^i=j\right)\right)}{\mu_{\eps,\vp^*}\left(H_\delta(\vp^*)\right)}\\
&=:\beta_{k}^{(2)}(j,\v,\v').
\end{split}
\end{equation*}
\end{thm}

\begin{rem}\label{rem:simdiff}
 Notice that in general, $S^{rec}\not\subset \tilde S^{rec}$, indeed one may consider the case $d=1$ with periodic points $(a_{\v},a_{-\v})$ in the sense that $(\tau a_{\v},\tau a_{-\v})=(a_{\v},a_{-\v})$, in that case, $(a_{\v},a_{-\v}) \in S^{rec}$ but $(a_{\v},a_{-\v}) \notin \tilde S^{rec}$. We also have $\tilde  S^{rec}\not\subset S^{rec}$. One can take $I^{\ZZ^2}$ and $\v \neq \v'$ such that $(a_{\v},a_{-\v})\in \tilde S^{rec}$ with $(\tau^3 a_{\v},\tau^3 a_{-\v})=(a_{-\v'},a_{\v'})$ such that $\tau a_{\v}$ and $\tau^2 a_{\v'}$ belong respectively to $A_{\eps,-\v'} \backslash a_{-\v'}$ and $A_{\eps,-\v} \backslash a_{-\v}$ and suppose that the orbit of $(\tau^i a_{\v})_{i\in \N}$ crosses no collision interval except for $\tau^3a_{\v}$. Then by construction, $( a_{\v}, a_{-\v})\in \tilde S^{rec}\backslash S^{rec}$.
\end{rem} 
\begin{rem}
By Remark \ref{rem:simdiff}, in general, $\beta_{k}^{(2)}(j,\v,\v')\not=\beta_{k}^{(1)}(j,\v,\v')$. Indeed, as a by-product of the proof of Theorem \ref{thm:main3} in section \ref{sec:proofthm3}, we obtain that $\beta_{k}^{(2)}(j,\v,\v')=0$ whenever $\tilde S^{rec}=\emptyset$ and $\beta_{k}^{(1)}(j,\v,\v')=0$ whenever $S^{rec}=\emptyset$. Finally, notice that $\beta_k^{(1)}(0,\v,\v')=q_k(\v,\v')$ with $q_k(\v,\v')$ as given in Theorem \ref{thm:main2}.
\end{rem}

\begin{rem}
Starting with the work of \cite{PS20}, see also \cite{PS21, Z22}, spatio-temporal Poisson processes were obtained by recording not only the successive times of visits to a set, but also the positions. It would be interesting to investigate if similar results can be obtained in infinite dimensional collision models as the one studied in this current work.   
\end{rem}
\subsection{An example}\label{sub:example} In this subsection we present an example that satisfies the above results and show how to compute $\theta$ of Theorem \ref{thm:main2} and the law for the random variables $(X_i)_{i\in \N}$ of Theorem \ref{thm:main3} in a concrete situation.
Let $d=1$ and $\tau :[0,1] \to [0,1]$ be given by $x\mapsto 5 x \,\mod 1$. Since $d=1$, we only have two collision sets $A_{\delta,\v}$, $\v\in V:=\{-1,1\}$. Assume that the collision sets are centered around $a_{\v}$, $\v\in V$, with  $a_1=\frac{1}{2}$  and $a_{-1}=\frac{1}{4}$. Notice that $a_1,a_{-1}$ are fixed points for the map $\tau$ ($\tau a_{\v}=a_\v$ for $\v\in V$). One can check that $S^{rec}=\{(a_1,a_{-1}),(a_{-1},a_1)\}$ and $\tilde S^{rec}=\emptyset$\footnote{If $(a_\v,a_{-\v})\in \tilde S^{rec}$ then there would be $k\in \NN$ such that $(\tau a_\v,\tau a_{-\v})=(a_{-\v},a_\v)$ which is impossible because $a_\v$ is a fixed point of $\tau$.}. We can now apply Theorem \ref{thm:main1} and Theorem \ref{thm:main2}: the first collision rate $-\ln(\lambda_{\eps,\p^*})$ and 
$$\lambda_{\varepsilon,\delta}=1-\mu_{\eps,\vp^*}(H_\delta(\vp^*))\cdot\theta(1+o(1)),$$ 
where
 \begin{align*}
\theta &=1-\sum_{\v,\v'\in V}\sum_{k\in {\KK(\v,\v')}}q_{k}(\v,\v')\\
&=1- \sum_{k\geq 1}q_k(1,-1)-\sum_{k\geq 1}q_k(-1,1)
\end{align*}
with
\begin{align*}
&q_{k}(\v,\v')=\nonumber\\
&\frac{1}{\sum_{(q,\v)\in \mathcal N_{\vp^*}}\rho_{\tau}(a_\v)\frac{\rho_{\eps,\q}(a_{-\v}^+)+\rho_{\eps,\q}(a_{-\v}^-)}{2}}\times\nonumber\\
    &\frac{\rho_{\tau}(a_{\v})}{|(\tau^{k+1})'(a_{\v})|}\left(\frac{1}{2|(\tau^{k+1})'(a_{-\v})|}\hat \rho_{\eps,\Lambda,k}(a_{-\v}^+)+\frac{1}{2|(\tau^{k+1})'(a_{-\v})|}\hat \rho_{\eps,\Lambda,k}(a_{-\v}^-)\right)\\
&=\frac{1}{\rho_{\tau}(a_1)\left(\frac{\rho_{\eps,\p^*+1}(a_{-1}^+)+\rho_{\eps,\p^*+1}(a_{-1}^-)}{2}\right)+\rho_{\tau}(a_{-1})\left(\frac{\rho_{\eps,\p^*+1}(a_{1}^+)+\rho_{\eps,\p^*+1}(a_{1}^-)}{2}\right)}\times\nonumber\\
    &\frac{\rho_{\tau}(a_{1})}{|(\tau^{k+1})'(a_{1})|}\left(\frac{1}{2|(\tau^{k+1})'(a_{-1})|}\hat \rho_{\eps,\Lambda,k}(a_{-1}^+)+\frac{1}{2|(\tau^{k+1})'(a_{-1})|}\hat \rho_{\eps,\Lambda,k}(a_{-1}^-)\right).
\end{align*}
We now compute $\hat{\rho}_{\eps,\Lambda,k}(a_{-\v}^\pm)$ for $k\geq 1$. We consider the case when the direction is $\v=+1$. The case when $\v=-1$ follows the same reasoning. To estimate  $\hat{\rho}_{\eps,\Lambda,k}(a_{-\v}^\pm)$ notice that for any $\vx$ such that $\pi_\q(\vx)$ belongs to a little neighborhood of $a_\v$, $\pi_\q(T_{\epsilon,\p^*}^j\vx)$ remains close to the periodic point $a_\v$ for any $j\leq k$. Thus, for any $j\leq k$ $\Psi_j^\q(\p^*+ 1)=\p^*+1$ and thus $1_{(\Psi_k^{\q}=\p^*+1) \cap \bigcap_{i=1}^r \,^c(\Psi_{j_i}^{\q}=\p^*+\w_{j_i})}(\vx)=0$ for $k\geq 2$ and $1_{(\Psi_k^{\q}=\p^*+1)}(\vx)=1$. Then, 
\begin{align*}
\hat{\rho}_{\eps,\Lambda,k}&=\E(1_{ (\Psi_k^{\q}=\p^*+\v') \cap \bigcap_{i=1}^r \,^c(\Psi_{j_i}^{\q}=\p^*+\w_{j_i})}\rho_{\eps,\Lambda}(.)|I^{\{q\}}).
\end{align*}
Consequently, for $k\ge 2$ $$\hat{\rho}_{\eps,\Lambda,k}=0$$
while 
$$\hat{\rho}_{\eps,\Lambda,1}=\rho_{\eps,\q}.$$
Thus, 
\begin{align*}
    q_{1}(1,-1)+q_{1}(-1,1)&=\frac{1}{|(\tau^{2})'(a_{1})||(\tau^2)'(a_{-1})|}\\
    &=5^{-(2+2)}=5^{-4}.
\end{align*}
We thus deduce the extremal index $\theta$
$$\theta=1-5^{-4}.$$
And the rare event statistics from Theorem \ref{thm:main2} follows, $\exists$ $C>0$, $\xi_\delta>0$, with $\lim_{\delta\to 0}\xi_\delta=1$, such that for all $t>0$
$$\left|\mu_{\eps,\vp^*}\Big\{t_\delta\ge\frac{t}{\xi_\delta\mu_{\eps,\vp^*}(H_\delta(\vp^*))}\Big\}-e^{-\theta t}\right|\le (t\vee 1)e^{-\theta t} C\delta^2\left|\ln(\delta^2)\right|.$$
For the compound Poisson process $Z(t):=\sum_{i=1}^{N(t)}X_i$ of Theorem \ref{thm:main3}
we can also conclude that (see the a detailed justification in subsection \ref{subsect:just})
\begin{equation}\label{eq:examplepoisson}
\tilde \theta(s)=1-\sum_{\v,\v'\in V}\sum_{k\in {\KK(\v,\v')}}\beta_{k,\delta}^{(1)}(j,\v,\v')=\theta
\end{equation}
and that the law of the integer variable $X_i$ is given by the following characteristic function $\phi_{X}(s)=e^{is}$. In other words, $X_i$ is almost surely constant and equal to $1$ and $Z(t)\sim \mathcal P(\theta t)$, that is a Poisson law of intensity $\theta t$.

\subsection{Strategy of the proofs} The proofs are based on spectral techniques of \cite{KL99,KL09',K12}, albeit in an infinite dimensional setting. We proceed as follows: 

\noindent$\bullet$ Show $\LG_{\eps, \vp^*}$ and $\hat\LG_{\eps,\delta, \vp^*}$ satisfy a uniform Lasota-Yorke inequality, and recalling \eqref{eq:collisionset}, show that the difference of the two, acting on measures from the strong space and evaluated in the weak norm, is $\mathcal O(\delta)$. Then show for $\eps$ small enough $\LG_{\eps,\vp^*}$ has a spectral gap on $\mathcal B$, and by \cite{KL99}, conclude for $\delta$ small enough that $\hat\LG_{\eps,\delta, \vp^*}$ also has a spectral gap, with leading eigenvalue $\lambda_{\eps,\delta}\in(0,1)$. 

\noindent$\bullet$ Using the notation $\lambda_{\eps,\delta}$ as above, we now differentiate $\lambda_{\eps,\delta}$ following the abstract results of \cite{KL09'} and obtain a formula for $\theta$, which is the main technical part of this step. To obtain the law for the first collision time, with sharp error term, we follow \cite{K12}.

\noindent$\bullet$ Finally, as suggested by \cite{AFGV24}, to obtain a limit law for counting the number of visits to a target set, one can still use the spectral framework of \cite{KL09'}.  For this purpose, we study the `twisted' transfer operator $\tilde{\mathcal{L}}_{\delta,s}(\cdot):=\mathcal{L_{\eps,\delta}}(e^{is1_{H_\delta(\p^*)}}\cdot)$, where $\mathcal{L_{\eps,\delta}}$ is the transfer operator of the fully coupled map lattice $T_{\eps,\delta}$. This will provide the expression of the characteristic function of \eqref{eq:counting} and its limit as $\delta\to0$.
\section{Proof of Theorem \ref{thm:main1}}\label{sec:pf1}
Recall the definition of the operator $\hat \LG_{\eps, \delta, \vp^*}$, which was introduced in \eqref{eq:ophole}
\begin{equation*}
    \int \varphi d\hat \LG_{\eps, \delta, \vp^*}(\mu)=\int \varphi \circ T_{\eps,\vp^*} 1_{X_{0,\delta}(\p^*)} d\mu.
\end{equation*}
Note that by its definition, $\hat \LG_{\eps, \delta, \vp^*}$ satisfies a uniform Lasota-Yorke inequality  (see Proposition \ref{propLasota-perturbed0}). But since the unit ball of $\|\cdot\|$ is \emph{not} compact in the $|\cdot|$ the Lasota-Yorke inequality does not immediately imply that $\hat \LG_{\eps, \delta, \vp^*}$ is quasi-compact. However, using perturbation arguments, below we show that $\hat \LG_{\eps, \delta, \vp^*}$ has in fact a spectral gap when acting on $\mathcal B$.
\subsection{Spectral gap for the rare event transfer operator}
\begin{lem}\label{lem:close}
$$\sup_{\|\mu\|\le 1}|(\LG_{\eps, \vp^*}-\hat \LG_{\eps,\delta, \vp^*})\mu |=\mathcal O({\delta}).$$ 
\end{lem}  
\begin{proof}
For $\mu\in \mathcal B$ and $\varphi\in\mathcal D_1$, which depends on a finite number of coordinates that are included in $\Lambda_0\subset I^{\ZZ^d}$, we have   
    \begin{align}
&\int  \varphi d\left[(\LG_{\eps,\delta}-\hat\LG_{\eps,\delta,\p^*})\mu\right]\nonumber\\
&=\sum_{(\q,\v)\in \mathcal N_{\p^*}}\int \varphi\circ T_{\eps,\p^*} 1_{A_{\delta,-\v}}\circ \pi_{\q} 1_{A_{\delta,\v}}\circ \pi_{\p^*} d\mu \nonumber\\
&\le\|\varphi\|_\infty \sum_{(\q,\v)\in \mathcal N_{\p^*}}\int_{I^\Lambda} 1_{A_{\delta,-\v}}\circ \pi_{\q} 1_{A_{\delta,\v}}\circ \pi_{\p^*} hdm_{\Lambda}\nonumber\\
&\le \|\varphi\|_\infty \sum_{(\q,\v)\in \mathcal N_{\p^*}}\int_{I^2} 1_{A_{\delta,-\v}}(x_{\q}) 1_{A_{\delta,\v}} (y_{\p^*}) h_{\v}(x_{\q},{y_{\p^*}})dx_\q dy_{\p^*},\label{eqlemclose}
\end{align}
where $\Lambda=\Lambda_0\cup \mathcal N_{\p^*}$, $h$ is the density of the marginal of $\mu$ on $\Lambda$ and $h_{\v}(x_{\q},{y_{\p^*}})= \int hdm_{\Lambda\setminus \mathcal N_{\p^*}\cup\{\p^*\}}$. Note that $|h_{\v}|_{BV}\le |h|_{BV}\le\|\mu\|<\infty$. Thus, by applying first the Cauchy-Schwarz inequality then the Sobolev inequality, we obtain 
$$\int  \varphi d\left[(\LG_{\eps,\p^*}-\hat\LG_{\eps,\delta,\p^*})\mu\right]\le 2d\|\varphi\|_{\infty}\|\mu\|\delta.$$
\end{proof}

\begin{proof}[Proof of Theorem \ref{thm:main1}]
Let $\LG_{\eps,\delta}$ denotes the transfer operator of the fully coupled systems; i.e.,
for $\varphi\in\mathcal D$ and $\mu$ a Borel complex measure
$$\int\varphi d\LG_{\eps,\delta}\mu=\int\varphi\circ T_{\eps,\delta}d\mu.$$
By the results of \cite{KL09, KL06}\footnote{In particular see equation (3.20) in \cite{KL06}. Note also that the $\theta$ in (3.20) of \cite{KL06} can be taken to be 1 for the coupling by collision model, similar to what was done in \cite{KL09}.} for $\eps>0$ sufficiently small, there exists\footnote{This is the first instance where the mixing of the site dynamics $\tau$ is needed.} a $C>0$ and $\gamma\in (0,1)$ such that for all $\mu\in\mathcal B$, with $\mu(1)=0$, $n\ge 0$,  
\begin{equation*}
\|\LG_{\eps,\delta}^n\mu\|\le C\gamma^n\|\mu\|.
\end{equation*}
Using the above inequality, the Lasota-Yorke inequality for $\LG_{\eps,\p^*}$ from Proposition \ref{propLasota-perturbed0}, and the fact that for any $\mu\in\mathcal B$
$$|(\LG_{\eps, \vp^*}^m- \LG_{\eps,\delta}^m)\mu |\leq C_m \delta\|\mu\|$$
we can also conclude, that for $\delta$ small enough, there exists a $C>0$ and $\gamma\in (0,1)$ such that for all $\mu\in\mathcal B$, with $\mu(1)=0$, $n\ge 0$,
\begin{equation}\label{eq:contraction}
\|\LG_{\eps, \vp^*}^n\mu\|\le C\gamma^n\|\mu\|.
\end{equation}
Then by a weak compactness argument (see \cite{KL05,KL06}) $T_{\eps,\p^*}$ admits an invariant probability measure $\mu_{\eps,\p^*}\in\mathcal B$. For any $\mu\in \mathcal B$, let  
$$\Pi_{1}\mu= \mu(1)\cdot \mu_{\eps,\p^*} \text{ and }Q=\LG_{\eps, \vp^*}-\Pi_1.$$
Notice that $\Pi_1^2=\Pi_1$ and $\Pi_1Q=Q\Pi_1=0$. Moreover, for any $\mu\in \mathcal B$
$$Q\mu=\LG_{\eps, \vp^*}\mu- \mu(1)\cdot \mu_{\eps,\p^*}=\LG_{\eps, \vp^*}(\mu-\mu(1)\cdot \mu_{\eps,\p^*}).$$
Therefore, by \eqref{eq:contraction}, 
$$\|Q^n\mu\|\le C\gamma^n\|\mu\|$$
and consequently $\text{spec}(\LG_{\eps, \vp^*})\cap\{|z|>\gamma\}=\{1\}$; i.e., $\LG_{\eps, \vp^*}$ admits a spectral gap on $\mathcal B$. Note that the spectral data of $\LG_{\eps, \vp^*}$ do not depend on $\delta$. Consequently, by Proposition \ref{propLasota-perturbed0} and Lemma \ref{lem:close}, for $\delta$ sufficiently small, the Keller-Liverani Stability result\footnote{See also \cite{Pe} section 7.4 for a reformulation of the stability result of \cite{KL99} which we use in this work.} \cite{KL99} implies that the operator $\hat\LG_{\eps,\delta, \vp^*}$ admits a spectral gap on $\mathcal B$. We denote the corresponding simple dominant eigenvalue by $\lambda_{\eps, \delta}$. We denote the corresponding eigenmeasure by $\hat\mu_{\eps,\delta, p^*}$. It is obvious that $\lambda_{\eps, \delta}\in (0,1)$, since for any positive measure $\mu$,
$$|\hat\LG_{\eps,\delta, \vp^*}^n\mu| =|\LG_{\eps,\vp^*}^n1_{X_{n,\delta}(\vp^*)}\mu|\le |1_{X_{n,\delta}(\vp^*)}\mu|<|\mu|.$$
Moreover, we can write 
\begin{equation}\label{eq:specdec}
\lambda_{\eps, \delta}^{-n}\hat\LG_{\eps,\delta, \vp^*}^n=\Pi_{\lambda_{\eps, \delta}} + \hat Q_{\eps,\delta, \vp^*}^n,
\end{equation}
with $\| \hat Q_{\eps,\delta, \vp^*}^n\|<1$. Since $\Pi_{\lambda_{\eps, \delta}}$ is a rank one operator 
$$
\Pi_{\lambda_{\eps, \delta}}=\Psi(.)\hat\mu_{\eps,\delta, p^*},
$$
where $\Psi(.)$ is a linear form on $\mathcal B$. Let $\nu\in \mathcal B$ be a positive measure and $\varphi\in\mathcal D$ be a positive function. By \eqref{eq:specdec}, we have
\begin{align*}
\left|\int \varphi d\lambda_{\eps, \delta}^{-n}\hat\LG_{\eps,\delta, \vp^*}^n\nu-\int \varphi d\Pi_{\lambda_{\eps, \delta}}\nu\right|&\leq |\varphi|_{\infty}\left\|\lambda_{\eps, \delta}^{-n}\hat\LG_{\eps,\delta, \vp^*}^n\nu-\Pi_{\lambda_{\eps, \delta}}\nu\right\|\\
&\underset{n\rightarrow\infty}{\longrightarrow} 0.
\end{align*}
Thus, $\int \varphi d\Pi_{\lambda_{\eps, \delta}}\nu$ is positive. In particular, $\int d\Pi_{\lambda_{\eps, \delta}} m_{\ZZ^d}$ is positive. Therefore, by definition of $\hat\LG_{\eps,\delta, \vp^*}$ and \eqref{eq:specdec} 
\begin{align*}
m_{\ZZ^d}(X_{n,\delta}(\vp^*))&=\int d\hat\LG_{\eps,\delta, \vp^*}^nm_{\ZZ^d}\\
&=\lambda^n_{\eps,\delta}\left[\int d\Pi_{\lambda_{\eps, \delta}} m_{\ZZ^d}+\int d Q_{\eps,\delta, \vp^*}^n m_{\ZZ^d}
\right],
\end{align*}
Consequently, using the above result and the fact that \\
$|Q_{\eps,\delta, \vp^*}|\le \|Q_{\eps,\delta, \vp^*}\|<1$, $\exists C>0$ such that
\begin{align*}
\left|\frac{1}{n}\ln \left(m_{\ZZ^d}(X_{n,\delta}(\vp^*))\right)-\ln \lambda_{\eps,\delta}\right|&=\frac 1 n\ln\left|\int d\Pi_{\lambda_{\eps, \delta}} m_{\ZZ^d}+\int d Q_{\eps,\delta, \vp^*}^n m_{\ZZ^d}\right|\\
&\le\frac{C}{n}.
\end{align*}
Thus, $\lim_{n \to\infty}\frac{1}{n}\ln m_{\ZZ^d}(X_{n,\delta}(\vp^*))=\ln\lambda_{\eps,\delta}$.
\end{proof}

\section{Proof of Theorem \ref{thm:main2}}\label{sec:pf2}
We first introduce the following notation
\begin{equation} \label{eq:eta}
\eta_{\delta}:=\sup_{\|\mu\|\le 1}\left|\left(\LG_{\eps,\p^*}(\mu)-\hat\LG_{\eps,\delta,\p^*}(\mu)\right)(1)\right|,
\end{equation}
Let $\mu_{\eps,\p^*}\in \mathcal B$ be the unique eigenmeasure corresponding to the eigenvalue 1 for the transfer operator $\LG_{\eps,\p^*}$, 
given a finite set $\Lambda\subset \ZZ^d$, we will denote by $\rho_{\eps,\Lambda}$ the density 
   of $\pi_{\Lambda}*\mu_{\eps,\vp^*}$.
We now set
\begin{equation}\label{eq:Delta}
\Delta_\delta:=\left(\LG_{\eps,\p^*}(\mu_{\eps,\p^*})-\hat\LG_{\eps,\delta,\p^*}(\mu_{\eps,\p^*})\right)(1).
\end{equation}
The proof of Theorem \ref{thm:main2} relies on a subtle application of \cite{KL09'}. To check the different assumptions of \cite{KL09'}, we first prove the following Lemma which provides a convenient approximation of the quantity $\Delta_\delta$ in terms of the Lebesgue measure. This will play later a key ingredient in the proof of Lemma \ref{lem:verirare} and Lemma \ref{lem:thetalemma}.

\begin{lem}\label{lemrefdenom}
The following holds:
\begin{enumerate}
\item The mass of $H_\delta(\vp^*)$ under $\mu_{\eps,\p^*}$ scales as follows:
    \begin{align*}
        \mu_{\eps,\p^*}(H_\delta(\vp^*))&\underset{\delta \to 0}\sim \sum_{(\q,\v)\in \mathcal N_{\vp^*}}{\rho_{\tau}(a_\v)m(A_{\delta,\v}})m(A_{\delta,-\v})\frac{\rho_{\eps,\q}(a_{\v}^+)+\rho_{\eps,\q}(a_{\v}^-)}{2},
    \end{align*}
    where $\rho_{\tau}=d\mu_\tau/dm$ and $\mu_\tau$ is the $\tau$-absolutely continuous invariant measure and $\rho_{\eps,\q}=d(\pi_{\q}*\mu_{\eps,p^*})/dm$, $\rho_{\eps,q}(a_{-\v}^+):=\lim_{x\to a_{-\v}^+}\rho_{\eps,\q}(x)$ and $\rho_{\eps,\q}(a_{-\v}^-)=\lim_{x\to a_{-\v}^{-}}\rho_{\eps,\q}(x)$. In particular,
$$\lim_{\delta\to 0}\frac{\mu_{\eps,\p^*}(H_\delta(\vp^*))}{m_{\ZZ^d}(H_\delta(\vp^*))}=\frac{1}{|\mathcal N_{\vp^*}|} \sum_{(q,\v)\in \mathcal N_{\vp^*}}\rho_{\tau}(a_\v)\frac{\rho_{\eps,\q}(a_{-\v}^+)+\rho_{\eps,\q}(a_{-\v}^-)}{2}.$$
\item In addition for $\eps>0$ small enough, 
$$\frac{1}{|\mathcal N_{\vp^*}|} \sum_{(q,\v)\in \mathcal N_{\vp^*}}\rho_{\tau}(a_\v)\frac{\rho_{\eps,\q}(a_{-\v}^+)+\rho_{\eps,\q}(a_{-\v}^-)}{2}>0.$$
\end{enumerate}
\end{lem}
\begin{proof}
Since $\mu_{\eps,\p^*}\in\mathcal B$, for any finite set\footnote{the expression $\mu \ll \nu$ means that $\mu$ is absolutely with respect to $\nu$.} $\Lambda \subset \ZZ^d$, $\pi_\Lambda*\mu_{\eps,\p^*} \ll \pi_\Lambda*m_{\ZZ^d}$. Thus
\begin{align}
    \mu_{\eps,\p^*}(H_\delta)&=\sum_{(\q,\v)\in \mathcal N_{\vp^*}}\mu_{\eps,\p^*}(A_{\delta,\v}(\q)\cap A_{\delta,-\v}(\vp^*))\label{eqdecouple1}\\
    &=\sum_{(\q,\v)\in \mathcal N_{\vp^*}}\int_{A_{\delta,\v}}\rho_{\tau}dm\int_{A_{\delta,-\v}}d\pi_{\q}*\mu_{\eps,\p^*}\label{eqdecouple2}\\ 
       &=\sum_{(\q,\v)\in \mathcal N_{\vp^*}}\mu_{\tau}(A_{\delta,\v})\int 1_{A_{\delta,-\v}}\rho_{\eps,\q}dm\\
       &
       =\sum_{(\q,\v)\in \mathcal N_{\vp^*}}\frac{\mu_{\tau}(A_{\delta,\v})}{m(A_{\delta,\v})}m(A_{\delta,\v})m(A_{\delta,-\v})\frac{1}{m(A_{\delta,-\v})}\int_{A_{\delta,-\v}}\rho_{\eps,\q}dm
       \\
       & \underset{\delta \to 0}\sim \sum_{(q,\v)\in \mathcal N_{\vp^*}}\rho_{\tau}(a_\v)m(A_{\delta,\v})m(A_{\delta,-\v})\frac{\rho_{\eps,\q}(a_{-\v}^+)+\rho_{\eps,\q}(a_{-\v}^-)}{2}.\label{eqdecouple3}
\end{align}
where in \eqref{eqdecouple3} we used the fact that the density $\rho_\tau$ is $C^1$ and the fact that since $\mu_{\eps,\p^*}\in\mathcal B$, the limits $\rho_{\eps,q}(a_{-\v}^+):=\lim_{x\to a_{-\v}^+}\rho_{\eps,\q}(x)$ and $\rho_{\eps,\q}(a_{-\v}^-)=\lim_{x\to a_{-\v}^{-}}\rho_{\eps,\q}(x)$ are well defined. Note that the passage from \eqref{eqdecouple1} to \eqref{eqdecouple2} is due to the decoupling of the measure $\mu_{\eps,\p^*}$ in $\vp^*$ (no collision with the site $\p^*$) which can be seen as $\pi_{\Lambda}*\mu_{\eps,\p^*} =\pi_{\Lambda\backslash\{\vp^*\}}*\mu_{\eps,\p^*}\otimes \mu_\tau$ with $\mu_\tau$ being the invariant measure associated to the site transformation $\tau$.

To prove the last statement of Lemma \ref{lemrefdenom}, i.e
\begin{align*}
\frac{1}{|\mathcal N_{\vp^*}|} \sum_{(q,\v)\in \mathcal N_{\vp^*}}\rho_{\tau}(a_\v)\frac{\rho_{\eps,\q}(a_{-\v}^+)+\rho_{\eps,\q}(a_{-\v}^-)}{2}>0
\end{align*}
we first prove the following two statements
\begin{itemize}
    \item[i)] There is a constant $C>0$ such that $\forall \eps>0$, $|\rho_{\eps,\q}|_{BV}\leq C$.
\item[ii)] Moreover, $\lim_{\eps \to 0}|\rho_\tau -\rho_{\eps,\q}|_{L^1(m)}=0$.
\end{itemize}

To prove statement i), notice that by Proposition \ref{propLasota-perturbed0}, $\exists\, C>0$, independent of $\eps$, such that $\|\mu_{\eps,\p^*}\|<C$. Hence, by definition of $\rho_{\eps,\q}$ we also have $|\rho_{\eps,\q}|_{BV}\le 
\|\mu_{\eps,\p^*}\|<C$.

We now prove the statement ii). Define $T_{\eps,\p^*,\q}=T_0\circ \Phi_{\eps,\p^*,\q}$ the map decoupled in $\p^*$ and $\q$ where for $r\in \ZZ^d$, 
\begin{equation}\label{decoupling2}
(\Phi_{\eps,\p^*,\q}(x))_r=\begin{cases}
(\Phi_{\eps,\p^*}(x))_r \quad \text{if } r-\q\notin V\cup\{0\}\\
       (\Phi_{\eps,\p^*}(x))_r \quad \text{if } \v=r-\q \text{ and } x_{r}\notin A_{\eps,-\v}\\
       x_r  \quad\quad\text{otherwise.}
       \end{cases}
\end{equation}

We denote by $\mathcal L_{\eps,\p^*,\q}$ the associated operator decoupled at sites $\p^*$ and $\q$. It satisfies the Lasota-Yorke inequality of Proposition \ref{propLasota-perturbed0}. We recall from \eqref{eq:contraction} we proved a spectral gap for $\mathcal L_{\eps,\p^*}$ with the following properties : 
given $\nu\in \B$,
$$
\mathcal L_{\eps,\p^*}\nu=\nu(1)\mu_{\eps,\p^*}+Q_{\eps,\p^*}(\nu).
$$
with $\|Q^n_{\eps,\p^*}\|<C\gamma^n$ for some $\gamma\in(0,1)$. The same strategy also holds for $\mathcal L_{\eps,\p^*,\q}$ : 
$$
\mathcal L_{\eps,\p^*,\q}\nu=\nu(1)\mu_{\eps,\p^*,\q}+Q_{\eps,\p^*,\q}(\nu).
$$
with $\|Q^n_{\eps,\p^*,\q}\|<C\gamma^n$ and $\mu_{\eps,\p^*,\q}$ the invariant measure for $T_{\eps,\p^*,\q}$.
Notice that $(\mu_{\eps,\p^*,\q})_{\q}=\rho_\tau$. Indeed, taking any local test function $\phi=\phi\circ \pi_\q$ at site $\q$, the invariant property of $\mu_{\eps,\p^*,\q}$ implies
\begin{align}
\int \phi d\mu_{\eps,\p^*,\q}=\int\phi \circ T_{\eps,\p^*,\q} d\mu_{\eps,\p^*,\q}&=\int \phi \circ \pi_q\circ T_{\eps,\p^*,\q}(\mu_{\eps,\p^*,\q})_{\q} dm\nonumber\\
&=\int  \phi \circ \tau (\mu_{\eps,\p^*,\q})_{\q} dm.    
\end{align}
Thus, $(\mu_{\eps,\p^*,\q})_{\q}$ is the unique invariant density for $\tau$ absolutely continue with respect to Lebesgue measure. Hence, it coincides with the invariant density $\rho_\tau$. Moreover,
 
\begin{align}
  \sup_{\|\mu\|\le 1}\left|\left((\LG_{\eps,\p^*}-\mathcal L_{\eps,\p^*,\q})\mu\right)(1)\right|\leq \eta_\eps.
\end{align}
 Therefore, by the Keller-Liverani \cite{KL99} stability result
\begin{align}
   \lim_{\eps\to 0} | \mu_{\eps,\p^*}-\mu_{\eps,\p^*,\q}| =0.
\end{align}
and by the linearity of conditioning over the state of all sites except $\q$, we deduce that
\begin{align*}
\lim_{\eps \to 0}|\rho_\tau -\rho_{\eps,\q}|_{L^1(m)}=0,    
\end{align*}
which concludes statement ii) above.
We now use 
(i) and (ii) to prove the
 $\rho_{\eps,\q}(a_{\v})>0$. 
To do so, we suppose that $\forall \eps >0$, $\rho_{\eps,\q}$ is continuous at $a_{\v}$ with $\rho_{\eps,q}(a_{\v})=0$ and then reach a contradiction.

Recall that $\inf_{x\in [0,1]}\rho_\tau(x)=\kappa>0$. We suppose that $\forall \eps >0$, $\rho_{\eps,\q}$ is continuous at $a_{\v}$ and $\rho_{\eps,\q}(a_{\v})=0$. Then statement i) implies that $\forall \beta ,\eps>0$, $\exists \gamma_{\beta}>0$ such that $\rho_{\eps,\q}(x)\leq \gamma_{\beta}$ for all $x$ s.t $|x-a_{\v}|\leq \beta$.
Because of statement ii), $|\rho_\tau-\rho_{\eps,\q}|_{L^1}=\xi_{\eps}=o(1)$, 
we obtain
\begin{align}
    \xi_{\eps}\geq |\int_{-\beta}^\beta \rho_\tau(x)dm-\int_{-\beta}^\beta \rho_{\eps,\q}dm|\geq 2\beta (\kappa -\gamma_\beta).
\end{align}
Consequently, fixing $\beta$ small enough so that $\kappa -\gamma_\beta>0$ we obtain a contradiction by letting $\eps\to 0$. Thus for $\eps>0$ small enough,
\begin{align*}
\frac{1}{|\mathcal N_{\vp^*}|} \sum_{(q,\v)\in \mathcal N_{\vp^*}}\rho_{\tau}(a_\v)\frac{\rho_{\eps,\q}(a_{-\v}^+)+\rho_{\eps,\q}(a_{-\v}^-)}{2}>0.
\end{align*}
\end{proof}
\begin{lem}\label{lem:verirare}
The following holds   
\begin{enumerate}
\item $\eta_\delta\to 0$ as $\delta\to 0$;
\item $\exists\, C>0$ such that $\eta_\delta\cdot\|(\LG_{\eps,\p^*}-\hat\LG_{\eps,\delta,\p^*})\mu_{\eps,\p^*}\|\le C |\Delta_\delta|$.
\end{enumerate}
\end{lem}
\begin{proof}
The first statement is a direct consequence of Lemma \ref{lem:close}, notice that we actually proved
    \begin{align}\label{eqcond1kl}
        \eta_\delta =\sup_{\mu \in \mathcal B} \frac{|\mu(H_\delta)(\p^*)|}{\|\mu\|}\leq 4d(m_{\ZZ^d}(H_\delta(\p^*)))^{1/2}.
    \end{align}
To prove the second statement, we first upper bound for $\|(\LG_{\eps,\p^*}-\hat\LG_{\eps,\delta,\p^*})\mu_{\eps,\p^*}\|$.

{Let $\tilde \phi\in \D_1$ a local coordinate map and let $\Lambda_0\subset \ZZ^d$ finite and $\varphi:I^{\Lambda_0}\mapsto \mathbb R$ such that $\tilde \phi=\varphi\circ\pi_{\Lambda_0}$ and let $\Lambda:=\Lambda_0\cup \{\p+\v, \p\in \Lambda_0, \v \in V\}$. Notice that, there is a map $\tilde T_{\eps,\p^*}:I^{\Lambda}\mapsto I^{\Lambda}$ such that $\pi_{\Lambda_0}\circ\tilde T_{\epsilon,\p^*}\circ \pi_{\Lambda}=\pi_{\Lambda_0}\circ T_{\eps,\p^*}$. Thus for $\varphi \in \D$ depending on local coordinates in $\Lambda_0$.

\begin{align*}
&\int \partial_r \varphi d\left[(\LG_{\eps,\p^*}-\hat\LG_{\eps,\delta,\p^*})\mu_{\eps,\p^*}\right]\\
&=\sum_{(\q,\v)\in \mathcal N_{\p^*}}\int \partial_r \varphi\circ \pi_{\Lambda_0}\circ T_{\eps, \p^*} 1_{A_{\delta,-\v}}\circ \pi_{\q} 1_{A_{\delta,\v}}\circ \pi_{\p^*} d\mu_{\eps,\p^*}\\
&=\sum_{(\q,\v)\in \mathcal N_{\p^*}}\int \partial_r \varphi\circ \pi_{\Lambda_0}\circ \tilde T_{\eps, \p^*} 1_{A_{\delta,-\v}}\circ \pi_{\q} 1_{A_{\delta,\v}}\circ \pi_{\p^*} \rho_{\eps,\p^*}\rho_{\eps,\Lambda\backslash \{\p^*\}}dm_{\Lambda}\\
&=\sum_{(\q,\v)\in \mathcal N_{\p^*}}\int \int \partial_r \varphi\circ \pi_{\Lambda_0}\circ \tilde T_{\eps, \p^*} 1_{A_{\delta,-\v}}\circ \pi_{\q} \rho_{\eps,\Lambda\backslash \{\p^*\}}dm_{\Lambda\backslash \{\p^*\}}  1_{A_{\delta,\v}}\rho_{\eps,\p^*}dm\\
\end{align*}
Recall that $\rho_{\eps,\Lambda_0}$ is the density of $\pi_{\Lambda_0}*\mu_{\eps,\p^*}$.
We now distinguish different cases according to whether $r=\p^*$ or not. Suppose first that $r\neq \p^*$,

\begin{align}
    &\int \partial_r \varphi\circ \pi_{\Lambda_0}\circ \tilde T_{\eps, \p^*} 1_{A_{\delta,-\v}}\circ \pi_{\q} \rho_{\eps,\Lambda\backslash \{\p^*\}}dm_{\Lambda\backslash \{\p^*\}}  1_{A_{\delta,\v}}\circ \pi_{\p^*}\rho_{\eps,\p^*}dm\nonumber\\
    &=\int 1_{A_{\delta,\v}}(\mu_{\eps,\p^*})_{\p^*}\int \partial_r \varphi\circ \pi_{\Lambda_0}\circ \tilde T_{\eps, \p^*} 1_{A_{\delta,-\v}}\circ \pi_{\q} \rho_{\eps,\Lambda\backslash \{\p^*\}}dm_{\Lambda\backslash \{\p^*\}}  dm.\label{eqref:phip}
\end{align}
Then one can introduce for $\vx'\in I^{\Lambda_0}$ the function $\varphi_{x_{\p^*}} (\vx'):=\varphi(\tau x_{\p^*},\vx'_{\neq \p^*})$ where we used the representation $\vx'=( \vx'_{\p^*},\vx'_{\neq \p^*})\in I\times I^{\Lambda_0 \backslash\{\p^*\}}\simeq I^{\Lambda_0}$. Notice that $\varphi \in \D_1$ and $\|\varphi\|_\infty\leq 1$, furthermore for any $\vx_1,\vx_2\in I^\Lambda_0$ such that $\pi_{\Lambda\backslash\{\p^*\}}(\vx_1)=\pi_{\Lambda\backslash\{\p^*\}}(\vx_2)$,  
\begin{align*}
\partial_r\varphi_{\vx_{\p^*}}\circ \pi_{\Lambda_0}\circ \tilde T_{\eps, \p^*}(\vx_1)&=\partial_r\varphi_{\vx_{\p^*}}\circ \pi_{\Lambda_0}\circ \tilde T_{\eps, \p^*}(\vx_2)\\
&=\int \partial_r\varphi_{\vx_{\p^*}}\circ \pi_{\Lambda_0}\circ \tilde T_{\eps, \p^*}(y,\pi_{\Lambda\backslash\{\p^*\}}(\vx_2)) \rho_{\eps,\p^*}(y)dm(y). 
\end{align*}
Plugging this relation in equation \eqref{eqref:phip},
\begin{align*}
    &\left|\int \partial_r \varphi\circ \pi_{\Lambda_0}\circ \tilde T_{\eps, \p^*} 1_{A_{\delta,-\v}}\circ \pi_{\q} \rho_{\eps,\Lambda\backslash \{\p^*\}}dm_{\Lambda\backslash \{\p^*\}} \right|\\
     &=\left|\int \partial_r \varphi_{\vx_{\p^*}}\circ \tilde T_{\eps, \p^*} 1_{A_{\delta,-\v}}\circ \pi_{\q} \rho_{\eps,\Lambda\backslash \{\p^*\}}dm_{\Lambda\backslash \{\p^*\}} \right|\\
    &=\left|\int \partial_r \varphi_{\vx_{\p^*}}\circ \pi_{\Lambda_0}\circ T_{\eps, \p^*} 1_{A_{\delta,-\v}}\circ \pi_{\q}d\mu_{\eps,\p^*} \right|\\
    &\leq \|\LG_{\eps,\p^*}(1_{A_{\delta,-\v}}\circ \pi_{\q}\mu_{\eps,\p^*})\|\\
    &\leq C \|\mu_{\eps,\p^*}\|,
\end{align*}
the last equation being a consequence of Proposition \ref{propLasota-perturbed0}. Since $\rho_{\eps,\p^*}$ is one dimensional bounded variation marginals, we deduce that $\|\rho_{\eps,\p^*}\|_{\infty}\leq \|\mu_{\eps,\p^*}\|$. We then deduce that, 

\begin{align}\label{eqpartialphibv}
    &\left|\int \partial_r \varphi\circ \pi_{\Lambda_0}\circ \tilde T_{\eps, \p^*} 1_{A_{\delta,-\v}}\circ \pi_{\q} \rho_{\eps,\Lambda\backslash \{\p^*\}}dm_{\Lambda\backslash \{\p^*\}}  1_{A_{\delta,\v}}\circ \pi_{\p^*}\rho_{\eps,\p^*}dm\right|\nonumber\\
&\leq \|\mu_{\eps,\p^*}\|\left|\int  1_{A_{\delta,\v}}\circ \pi_{\p^*}\rho_{\eps,\p^*}dm\right|\nonumber\\
&\leq \|\mu_{\eps,\p^*}\|\left|\int  1_{A_{\delta,\v}} \|\mu_{\eps,\p^*}\|dm\right|\nonumber\\
&\leq \|\mu_{\eps,\p^*}\|^2\left|\int  1_{A_{\delta,\v}}dm\right|\nonumber\\
&\leq \|\mu_{\eps,\p^*}\|^2\delta.
\end{align}

Now if one suppose $r=p^*$, for $\vx\in I^{\Lambda}$, we can fix the local map defined for $\vx'_{\p^*}\in I$ as
$$\varphi_{\vx_{\neq \p^*}}(\vx'_{\p^*}):=\varphi(\vx'_{\p^*},\pi_{\Lambda_0\backslash \{\p^*\}}(\tilde T_{\eps,\p^*}\vx_{\neq \p^*})).$$ Notice that $(\pi_{\Lambda_0}\circ \tilde T_{\eps, \p^*} (\vx'))_{\p^*}=\tau\vx'_{\p^*}$ for $\vx\in I^{\Lambda}$
then the same reasoning as when $r\neq \p^*$ holds,

\begin{align*}
    &\left|\int \partial_{\p^*} \varphi\circ \pi_{\Lambda_0}\circ \tilde T_{\eps, \p^*} 1_{A_{\delta,-\v}}\circ \pi_{\q} \rho_{\eps,\Lambda\backslash \{\p^*\}}dm_{\Lambda\backslash \{\p^*\}}  1_{A_{\delta,\v}}\circ \pi_{\p^*}\rho_{\eps,\p^*}dm\right|\\
       =&\left|\int \partial_{\p^*} \varphi_{\vx_{\neq \p^*}} (\tau x)  1_{A_{\delta,\v}}(x)\rho_{\eps,\p^*}(x)dm(x) 1_{A_{\delta,-\v}}\circ \pi_{\q} \rho_{\eps,\Lambda\backslash \{\p^*\}}dm_{\Lambda\backslash \{\p^*\}}(\vx_{\neq\p^*})\right|\\
    &\leq \left|\int \|\LG_\tau(1_{A_{\delta,\v}}\mu_\tau)\| 1_{A_{\delta,-\v}}\circ \pi_{\q} \rho_{\eps,\Lambda\backslash \{\p^*\}}dm_{\Lambda\backslash \{\p^*\}}\right|\\
   & \leq C\|\mu_\tau\|\int 1_{A_{\delta,-\v}}\circ \pi_{\q} \rho_{\eps,\Lambda\backslash \{\p^*\}}dm_{\Lambda\backslash \{\p^*\}}\\
   & \leq C\|\mu_\tau\|\int 1_{A_{\delta,-\v}} \rho_{\eps,\q}dm_{\q}\\
    &\leq  \|\mu_{\eps,\p^*}\|^2\delta,
\end{align*}
where we also used the fact that $(\rho_{\eps,\q})$ is a one dimensional map and thus $\|\rho_{\eps,\q}\|_{\infty}\leq \|\mu_{\eps,\p^*}\|$.
}
Therefore,
\begin{align*}
    \left\|(\LG_{\eps,\p^*}-\hat\LG_{\eps,\delta,\p^*})\mu_{\eps,\p^*}\right\| \leq 2\|\mu_{\eps,\p^*}\|^2\delta,
\end{align*}
and the above inequality together with \eqref{eqcond1kl} gives us, by Lemma \ref{lemrefdenom}, the second statement of the lemma.
\end{proof}
\subsection{Existence of $\theta$}
We now have all the requirements for the abstract framework of \cite{KL09'}: (A1)-(A4) of \cite{KL09'} follow from Proposition \ref{propLasota-perturbed0} and Theorem \ref{thm:main1}, while (A5)-(A6) of \cite{KL09'} follow from Lemma \ref{lem:verirare}.
We introduce the following notation\footnote{$\q_{k,\delta}$ depends on $\eps$, but we drop it from its notation, since $\eps$ will be kept fixed at this point and we will only consider limits in $\delta$.}
\begin{align}\label{eq:qex}
    q_{k,\delta}:=\frac{\mu_{\eps,\vp^*}\left(H_\delta(\vp^*)\cap \bigcap_{i=1}^k T_{\eps,\vp^*}^{-i}X_{0,\delta}(\vp^*)\cap T_{\eps,\vp^*}^{-(k+1)}H_\delta(\vp^*)\right)}{\mu_{\eps,\vp^*}\left(H_\delta(\vp^*)\right)}.
\end{align}
and let 
\begin{equation}\label{eq:qlimit}
q_k:=\lim_{\delta \to 0}q_{k,\delta}.    
\end{equation} 
If the previous limit exists we then set
\begin{equation}\label{eq:thetak}
\theta=1-\sum_{k=0}^{\infty}q_k
\end{equation}
 and, following the abstract framework of \cite{KL09'}, we obtain
\begin{equation}\label{eq:KLtheta}
\lim_{\delta\to 0}\frac{1-\lambda_{\eps,\delta}}{\mu_{\eps,\p^*}(H_{\delta}(\p^*)}=\theta.
\end{equation}
In this subsection we prove, for any $k\in \N$, that the limit in \eqref{eq:qlimit} exists and we find $\theta$ explicitly. Notice that for fixed $k$ the characteristic function $$1_{H_\delta(\vp^*)\cap \bigcap_{i=1}^k T_{\eps,\vp^*}^{-i}X_{0,\delta}(\vp^*)\cap T_{\eps,\p^*}^{-(k+1)}H_\delta(\vp^*)}$$ 
   depends only on the $k+1$ coordinates around the site $\vp^*$. Thus, we can 
   fix the box $\Lambda:=\vp^* + \{-(k+1),\dots,(k+1)\}^d$ around the vector $\vp^*$ and let $\rho_{\eps,\Lambda}$ be the density 
   of $\pi_{\Lambda}*\mu_{\eps,\vp^*}$.  Using the above and equation \eqref{eq:qex}, we  can rewrite $q_{k,\delta}$ as follows: 
\begin{align}\label{eq:qex1}
    q_{k,\delta}=\frac{1}{\mu_{\eps,\vp^*}\left(H_\delta(\vp^*)\right)}\int_{I^\Lambda}1_{H_\delta(\vp^*)\cap \bigcap_{i=1}^k T_{\eps,\vp^*}^{-i}X_{0,\delta}(\vp^*)\cap T_{\eps,\p^*}^{-(k+1)}H_\delta(\vp^*)} \rho_{\eps,\Lambda}dm_{\Lambda},
\end{align}
We have already dealt with the denominator of the expression for $q_{k,\delta}$ in \eqref{eq:qex1} in Lemma \ref{lemrefdenom}.

We introduce the following notation, for any subset $A \subset \ZZ^d$,
\begin{align*}
    (\Psi_k^{\q}\in A):=\{\vx\in I^{\mathbb Z^d}, \Psi_k^{\q}(\vx)\in A\}.
\end{align*}
Notice that the set $(\Psi_k^{\q}\neq \q')=(\Psi_k^{\q}\in \ZZ^d\setminus\{\q'\})$. We now prove that for any $(\q,\v),(\q',\v') \in \mathcal{N}_{\p^*}$, the sets 
$$A_{\delta,-\v}(\q)\cap T_{\eps,\p^*}^{-(k+1)}A_{\delta,-\v'}(\q')\cap (\Psi_k^{\q}\neq \q')$$ have negligible weight. 
\begin{lem}\label{lemnegnonrec}
For  any $(\q,\v) \in \mathcal N_{\vp^*}$ and $k>0$,
\begin{align}\label{eqneglreturn}
\lim_{\delta \to 0}\frac{1}{\mu_{\eps,\p^*}(A_{\delta,\v}(\q))}\mu_{\eps,\p^*}\left(A_{\delta,\v}(\q)\cap T_{\eps,\vp^*}^{-k}A_{\delta,\v}(\q)\cap (\Psi_k^{\q}\neq \q)\right)=0
\end{align}
    In addition for $(\q',\v')\in \mathcal N_{\vp^*}$,
\begin{align*}
\lim_{\delta \to 0}\frac{1}{\mu_{\eps,\p^*}(A_{\delta,\v}(\q))}\mu_{\eps,\p^*}\left(A_{\delta,\v}(\q)\cap T_{\eps,\vp^*}^{-k}A_{\delta,\v'}(\q')\cap (\Psi_k^{\q}\neq \q')\right)=0. 
\end{align*}
\end{lem}

\begin{proof}
    The proof of equation \eqref{eqneglreturn} relies on the fact that if $\vx \in(\Psi_k^{\q}\neq \q )\cap T_{\eps,\vp^*}^{-k}A_{\delta,\v}(\q)$,
    then it means that there is some coordinate $r\in \ZZ^d$ such that $\tau^k x_r \in A_{\delta,-\v}$. 
  The full event is then a subset of $\pi_{\q}^{-1}A_{\delta,\v}\cap \pi_{r}^{-1}(\tau^{-k}A_{\delta,-\v})$.
Therefore, it remains to show that
    \begin{align*}
\mu_{\eps,\p^*}(\pi_{\q}^{-1}A_{\delta,\v}\cap \pi_{r}^{-1}(\tau^{-k}A_{\delta,-\v}))&=o(\mu_{\eps,\p^*}(A_{\delta,\v}(\q))).
    \end{align*}
Since the density of $\pi_{\{\q,r\}}*\mu_{\eps,\p^*}$ is not necessarily bounded, the measure above does not break easily into a product of measures, thus we prove the above control through what follows. We denote by $\rho(x,y):=\frac{d\mu_{\eps,\p^*}}{dm_{\{\p^*,\q\}}}$, let $G_\delta(y):= \int \rho(x,y) \frac{1_{A_{\delta,\v}}}{m(A_{\delta,\v})}dx$. Then one can choose $\phi: I^{\{\p^*,\q\}}\mapsto \mathbb{R}$ such that $\partial_x\phi(x,y):=\frac{1_{A_{\delta,\v}}}{m(A_{\delta,\v})}(x)$. Notice then that $\int |\int \rho(x,y) \frac{1_{A_{\delta,\v}}}{m(A_{\delta,\v})}dx|dy\leq \|\rho\|$.

Since $B(0,\|\rho\|)$ is relatively compact in $\left(L^\infty(X)\right)'=\{\mu\in \B,\mu \ll Leb\}$ for the weak$^*$ topology. Therefore, the sequence $G_\delta$ admits a subsequence that converges to an element $G_0\in \left(L^\infty(X)\right)'$.\\
Thus, for any $\delta_0>0$ $\int G_\delta 1_{B_{\delta_0}}dy \underset{\delta \to 0}\to \int 1_{B_{\delta_0}}G_0(dy)$. And since for any $\delta \leq \delta_0$,  $\int G_\delta 1_{B_{\delta_0}}dy\geq  \int G_\delta 1_{B_{\delta}}dy$ we deduce that $\limsup\int G_\delta 1_{B_{\delta}}dy\leq  \int 1_{B_{\delta_0}}dG_0$ for any $\delta_0>0$. Consequently, $\limsup\int G_\delta 1_{B_{\delta}}dy=0$ since $G_0(dy)$ is a positive measure absolutely continuous with respect to the Lebesgue measure. 
\end{proof}
For each\footnote{Recall that $S^{rec}$ is defined in \eqref{eq:recset}. It is specific to the site $\p^*$.} $(a_{\v},a_{-\v})\in S^{rec}$, let $\mathbb K(\v,\v')$ be the set of integers $k$ such that $(\tau^{k+1}a_{\v},\tau^{k+1}a_{-\v})=(a_{\v'},a_{-\v'})$. To each such $k$ we associate the set $J_k$ composed of $j_1\leq \dots \leq j_r\leq k$, the recurrent time into $S^{rec}$ up to time $k$ (i.e , $(\tau^{j_i}(a_{\v}),\tau^{j_i}(a_{-\v}))=(a_{\w_{j_i}},a_{-\w_{j_i}})$).

\begin{lem}\label{lem:thetalemma}
For any $k\in \N$, the limit in \eqref{eq:qlimit} exists :
\begin{equation*}
q_k:=\lim_{\delta \to 0}q_{k,\delta}.    
\end{equation*} 
If
$S^{rec}= \emptyset$ then $\theta=1$; otherwise $\theta \in(0,1)$ and is given by the following formula   
\begin{align*}
\theta=1-\sum_{\v,\v'\in V}\sum_{k\in {\KK(\v,\v')}}q_{k}(\v,\v'),
\end{align*}
where $q_{k}(\v,\v')$ is non zero whenever $(\tau^{k+1}(a_{\v}),\tau^{k+1}(a_{-\v}))=(a_{\v'},a_{-\v'})$ and is defined as follows : 

\begin{align*}
&q_{k}(\v,\v')=\frac{1}{\sum_{(q,\v)\in \mathcal N_{\vp^*}}\rho_{\tau}(a_\v)\frac{\rho_{\eps,\q}(a_{-\v}^+)+\rho_{\eps,\q}(a_{-\v}^-)}{2}}\times\nonumber\\
    &\frac{\rho_{\tau}(a_{\v})}{|(\tau^{k+1})'(a_{\v})|}\left(\frac{1}{2|(\tau^{k+1})'(a_{-\v})|}\tilde \rho_{\eps,\Lambda,k}(a_{-\v}^+)+\frac{1}{2|(\tau^{k+1})'(a_{-\v})|}\tilde \rho_{\eps,\Lambda,k}(a_{-\v}^-)\right),
\end{align*}
where
$$
\tilde{\rho}_{\eps,\Lambda,k}=\E(1_{ (\Psi_k^{\q}=\p^*+\v') \cap \bigcap_{i=1}^r \,^c(\Psi_{j_i}^{\q}=\p^*+\w_{j_i})}\rho_{\eps,\Lambda}(.)|I^{\{\q\}}).
$$
\end{lem}
\begin{proof}
We first consider the case when $S^{rec}=\emptyset$. Then by the expression \eqref{eq:qex}, $\lim_{\delta \to 0}q_{k,\delta}=0$. Therefore, $q_k=0$ for all $k \geq 0$ and $\theta=1$ in this case. We now consider the case when $S^{rec}\not=\emptyset$ and we prove that there is $k\in \N^*$ for which $\limsup_{\delta \to 0}q_{k,\delta}>0$. 

Recall that we already computed the expression of the denominator of $q_{k,\delta}$ in Lemma \ref{lemrefdenom} whose expression depends on the term $m_{\ZZ^d}(H_\delta(\p^*))$. To focus on the numerator we now introduce the quantity $\check q_{k,\delta}$ with $q_{k,\delta}=\frac{m_{\ZZ^d}(H_\delta(\p^*))}{\mu_{\eps,\p^*}(H_\delta(\p^*))}\check q_{k,\delta}$ and prove the existence of its limit when $\delta \to 0$. Using the fact that $H_\delta(\p^*)=\bigcup_{(\q,\v)\in \mathcal N_{\vp^*}}A_{\delta,\v}(\p^*)\cap A_{\delta,-\v}(\q)$, we get 
\begin{align}\label{eq:tildeqvv}
    \check q_{k,\delta}&:=\frac{1}{m_{\ZZ^d}(H_\delta(\p^*))}\nonumber\\
    &\sum_{\v\in V}\mu_{\eps,\p^*}\left(A_{\delta,\v}(\p^*)\cap A_{\delta,-\v}(\q) \cap T_{\eps,\p^*}^{-(k+1)}H_{\delta}(\p^*) \cap \bigcap_{i=1}^k T_{\eps,\vp^*}^{-i}X_{0,\delta}(\vp^*)\right)\nonumber\\
    & =\sum_{\v,\v'\in V^{+}}\hat q_{k,\delta}(\v,\v'),
\end{align}
where 
\begin{equation}\label{eqqvnum}
  \hat q_{k,\delta}(\v,\v')= \frac{1}{m_{\ZZ^d}(H_\delta(\p^*))}\mu_{\eps,\p^*}\left(G_{\v,\v'}\right),
\end{equation}
and 
$G_{\v,\v'}$ is a short notation for
\begin{align}
A_{\delta,\v}(\p^*)\cap A_{\delta,-\v}(\q) \cap T_{\eps,\p^*}^{-(k+1)}\left(A_{\delta,\v'}(\p^*)\cap A_{\delta,-\v'}(\q')\right) \cap \bigcap_{i=1}^k T_{\eps,\vp^*}^{-i}X_{0,\delta}(\vp^*). \label{eqqvvvprime}
\end{align}
where $\vq'$ stands for the element $\vp^*+\v'$.

According to Lemma \ref{lemnegnonrec}, $\lim_{\delta \to 0}\frac{1}{m_{\ZZ^d}(H_\delta(\p^*))}\mu_{\eps,\p^*}\left(G_{\v,\v'}\right)$ can be reduced to
\begin{align}
    &\lim_{\delta \to 0}\frac{1}{m_{\ZZ^d}(H_\delta(\p^*))}\mu_{\eps,\p^*}\left(G_{\v,\v'}\right)\nonumber\\
    &=\lim_{\delta \to 0}\frac{1}{m_{\ZZ^d}(H_\delta(\p^*))}\mu_{\eps,\p^*}\left(A_{\delta,\v}(\p^*)\cap\right.\nonumber\\
    &\left. T_0^{-(k+1)}A_{\delta,\v'}(\p^*) \cap A_{\delta,-\v}(\q)\cap T_0^{-(k+1)}A_{\delta,-\v'}(\q)\cap (\Psi_k^{\q}=\q)\cap\right.\nonumber\\
    &\left. \bigcap_{i=1}^r\, ^c\left(A_{\delta,\v}(\p^*)\cap T_0^{-j_i}A_{\delta,\w_{j_i}}(\p^*)\cap A_{\delta,-\v}(\q)\cap T_0^{-j_i}A_{\delta,-\w_{j_i}}(\q)\cap (\Psi_{j_i}^{\q}=\p^*+\w_{j_i})\right)\right).\label{eqtoneglig}
\end{align} 
Indeed, Lemma \ref{lemnegnonrec} implies that in the expression of $G_{\v,\v'}$ the respective sets $A_{\delta,\v}(\p^*)\cap T_{\eps,\p^*}^{-(k+1)}A_{\delta,\v'}(\p^*)$ can be replaced by\footnote{The latter corresponds to  $A_{\delta,-\v}(\q)\cap T_{\eps,\p^*}^{-k}A_{\delta,-\v'}(\q')\cap (\Psi_k^{\q}=\q')$ since $\pi_{\q'}\circ T_{\eps,\p^*}^{k+1}(\vx)=\tau^{k+1} x_\q=\pi_q \circ T_0^{k+1}(\vx)$.} $A_{\delta,-\v}(\q)\cap T_0^{-(k+1)}A_{\delta,-\v'}(\q')\cap (\Psi_k^{\q}=\q')$ and $T_{\eps,\p^*}^{-j}(H_\delta (\p^*))$ by
$$\bigcap_{i=1}^r\, ^c\left(A_{\delta,-\v}(\q)\cap T_0^{-j_i}A_{\delta,\w_{j_i}}(\q')\cap (\Psi_{j_i}^{\q}=\p^*+\w_{j_i})\right).$$
In particular Lemma \ref{lemnegnonrec} implies that the integers $k$ such that $\lim_{\delta}\tilde q_{k\delta}$ is non zero are all contained in $\mathbb{K}(\v,\v')$, thus
the expression of $\theta$ can be simplified into 
\begin{equation}\label{eq:simplified}
\theta=1-\sum_{\v,\v'\in V}\sum_{\KK(\v,\v')}\lim_{\delta\to 0} \frac{m_{\ZZ^d}(H_\delta(\p^*))}{\mu_{\eps,\p^*}(H_\delta(\p^*))} \hat q_{k,\delta}(\v,\v').  
\end{equation}
Note that since $\tau$ is a uniformly expanding map, the set $T_0^{-(k+1)}A_{\delta,-\v'}(\q)$ is included in $T_0^{-j_i}A_{\delta,\w_{j_i}}(\q)$ for any $1\leq i \leq r$. Thus 
\begin{align}
    \lim_{\delta \to 0}&\frac{1}{m_{\ZZ^d}(H_\delta(\p^*))}\mu_{\eps,\p^*}\left(G_{\v,\v'}\right)=\nonumber\\
=&\lim_{\delta \to 0}\frac{1}{m_{\ZZ^d}(H_\delta(\p^*))}\nonumber\\
&\mu_{\eps,\p^*}\left(A_{\delta,\v}(\p^*)\cap T_0^{-(k+1)}A_{\delta,\v'}(\p^*) \cap A_{\delta,-\v}(\q)\cap T_0^{-(k+1)}A_{\delta,-\v'}(\q)\right.\nonumber\\
&\left.\cap (\Psi_k^{\q}=\q') \cap \bigcap_{i=1}^r\, ^c(\Psi_{j_i}^{\q}=\p^*+\w_{j_i})\right).\label{eqrecurmeasur}
\end{align}
Since the measure $\mu_{\eps,\p^*}$ is decoupled at site $\p^*$, for $\delta$ small enough we can split the numerator of $\hat q_{k,\delta}(\v,\v')$ into two factors. 

\begin{align}
   & \lim_{\delta \to 0}\hat q_{k,\delta}(\v,\v')=\lim_{\delta \to 0}\frac{1}{m_{\ZZ^d}(H_\delta(\p^*))}\mu_{\eps,\p^*}\left(A_{\delta,\v}(\p^*)\cap T_0^{-(k+1)}A_{\delta,\v'}(\p^*)\right) \label{eqcentre1}\\
    &\mu_{\eps,\p^*}\left(A_{\delta,-\v}(\q)\cap T_0^{-(k+1)}A_{\delta,-\v'}(\q)\cap (\Psi_k^{\q}=\q') \cap \bigcap_{i=1}^r\, ^c(\Psi_{j_i}^{\q}=\p^*+\w_{j_i})\right)\label{eqcentre2}
\end{align}
Now we will treat the two quantities \eqref{eqcentre1} and \eqref{eqcentre2} separately starting with the quantity \eqref{eqcentre1}: 
\begin{align}
\mu_{\eps,\p^*}\left(A_{\delta,\v}(\p^*)\cap T_0^{-(k+1)}A_{\delta,\v'}(\p^*)\right)&=\int_I 1_{A_{\delta,\v}\cap \tau^{-(k+1)}A_{\delta,\v'}}\rho_{\tau}dm\nonumber\\
&\sim  \frac{\rho_{\tau}(a_{\v})}{|(\tau^{k+1})'(a_{\v})|} m(A_{\delta,\v}).\label{eqmucentre}
\end{align}
where we recall that $\rho_{\tau}(.)$ is the invariant density of the site dynamics $(I,\tau)$. Now we treat the factor \eqref{eqcentre2}. Since $\mu_{\eps,\p^*}$ has bounded variation, $\pi_{\Lambda}*\mu_{\eps,\p^*}$ is absolutely continuous with respect to $m_\Lambda$. Thus we only need to look at the following quantity :
\begin{align}
    \mu_{\eps,\p^*}&\left(A_{\delta,-\v}(\q)\cap T_0^{-(k+1)}A_{\delta,-\v'}(\q)\cap (\Psi_k^{\q}=\q') \cap \bigcap_{i=1}^r\, ^c(\Psi_{j_i}^{\q}=\p^*+\w_{j_i})\right)\nonumber\\
    &=\int 1_{A_{\delta,-\v}(\q)\cap T_0^{-(k+1)}A_{\delta,-\v'}(\q)}1_{(\Psi_k^{\q}=\q') \cap \bigcap_{i=1}^r\, ^c(\Psi_{j_i}^{\q}=\p^*+\w_{j_i})}\rho_{\eps,\Lambda}(.)dm
\end{align}
Notice that the function $1_{(\Psi_k^{\q}=\q') \cap \bigcap_{i=1}^r\, ^c(\Psi_{j_i}^{\q}=\p^*+\w_{j_i})}\rho_{\eps,\Lambda}(.)$ is of bounded variation. Therefore,
$$
\hat \rho_{\eps,\Lambda,k}=\E_{m_\q}(1_{ (\Psi_k^{\q}=\q') \cap \bigcap_{i=1}^r\, ^c(\Psi_{j_i}^{\q}=\p^*+\w_{j_i})}\rho_{\eps,\Lambda}(.)|I^{\{q\}})
$$
is a one dimensional bounded variation map and thus is cad-lag.  
 Thus,
\begin{align}
&\mu_{\eps,\p^*}\left( A_{\delta,-\v}(\q)\cap T_0^{-(k+1)}A_{\delta,-\v'}(\q)\cap (\Psi_k^{\q}=\q') \cap \bigcap_{i=1}^r\, ^c(\Psi_{j_i}^{\q}=\p^*+\w_{j_i})\right)\nonumber\\
&=\int1_{A_{\delta,-\v}\cap \tau^{-(k+1)}A_{\delta,-\v'}} \hat \rho_{\eps,\Lambda,k}(.)dm\nonumber\\
&\sim \hat \rho_{\eps,\Lambda,k}(a_{-\v}^+)m(A_{\delta,-\v}^+\cap \tau^{-(k+1)}A_{\delta,-\v'})+\hat \rho_{\eps,\Lambda,k}(a_{-\v}^-)m(A_{\delta,-\v}^-\cap \tau^{-(k+1)}A_{\delta,-\v'}).\label{eqestimatneighbor}
\end{align}
Where $A_{\delta,-\v}^+:=A_{\delta,-\v}\cap \{y,y\geq a_{-\v}\}$ (resp $A_{\delta,-\v}^-:=A_{\delta,-\v}\cap \{y,y\leq a_{-\v}\}$) and $\hat \rho_{\eps,\Lambda,k,s}(a_{-\v}^\pm)=\lim_{x\to a_{-\v}^{\pm}} \rho_{\eps,\Lambda,k,s}(x)$.\\
Notice that $\tau^{k+1}$ is piecewise $C^2$ and piecewise onto. Therefore, for $\delta >0$, $m(A_{\delta,\v}^-\cap \tau^{-(k+1)}A_{\delta,\v'})=m(A_{\delta,\v}^-\cap \tau^{-(k+1)}A_{\delta,\v'}^+)$ or $m(A_{\delta,\v}^-\cap \tau^{-(k+1)}A_{\delta,\v'})=m(A_{\delta,\v}^-\cap \tau^{-(k+1)}A_{\delta,\v'}^-)$.\\
Since $\tau^{k+1}$ is a $C^1$ map at $a_{\v}$ and $a_{-\v}$ then 
\begin{align*}
    m(A_{\delta,\v}^-\cap \tau^{-(k+1)}A_{\delta,\v'})&=m(A_{\delta,\v}^+\cap \tau^{-(k+1)}A_{\delta,\v'})\\
    &=\frac 1 {2|(\tau^{k+1})'(a_{\v})|} m(A_{\delta,\v'}).
\end{align*}
The above is justified as follows: first for $\delta$ small enough, $\tau^{-(k+1)}A_{\delta,\v'}^{\pm} \subset A_{\delta,\v}^{\pm} $ and there is a unique inverse branch for $\tau^{k+1}$ between $A_{\delta,\v}^{\pm}$ and $\tau^{-(k+1)}A_{\delta,\v'}^\pm$ that we denote $\tau_*$ and $I_m$ the associated interval in the partition $(I_j)_{j}$.
\begin{align*}
    m(A_{\delta,\v}^\pm\cap \tau^{-(k+1)}A_{\delta,\v'}^\pm)&=\int_{I_j} 1_{A_{\delta,\v'}^\pm}\circ \tau^{k+1}dm\\
    &=\int_I 1_{A_{\delta,\v'}^\pm}\frac 1 {|(\tau^{k+1})'\circ \tau_*|}dm\\
    &\sim\frac 1 {|(\tau^{k+1})'(a_\v)|}m(A_{\delta,\v'}^\pm)\sim\frac 1 {2|(\tau^{k+1})'(a_\v)|}m(A_{\delta,\v'}).    
\end{align*}
By plugging the above into equation \eqref{eqestimatneighbor} we get the estimate of the factor \eqref{eqcentre1}:
\begin{align*}
\mu_{\eps,\vp^*}&(A_{\delta,-\v}(\q)\cap T_0^{-(k+1)}A_{\delta,-\v'}(\q')\cap (\Psi_k^{\q}=\q') \cap \bigcap_{i=1}^r\, ^c(\Psi_{j_i}^{\q}=\w_{j_i}))\\
&\sim\frac 1 {2|(\tau^{k+1})'(a_\v)|} \left(\hat \rho_{\eps,\Lambda,k}(a_{-\v}^+)+\hat \rho_{\eps,\Lambda,k}(a_{-\v}^-)\right)m(A_{\delta,\v'}).
\end{align*}
We then deduce from the expression of $\hat q_{k,\delta}$ in \eqref{eqcentre1} and the respective behaviour of the term \eqref{eqcentre1}, \eqref{eqcentre2} found in \eqref{eqmucentre} that $\hat q_{k}(\v,\v'):=\lim_{\delta \to 0} \hat q_{k,\delta}(\v,\v')$ is well defined and explicitly given by 
\begin{align}
&\hat q_{k}(\v,\v')=\frac{1}{|\mathcal N_{\p^*}|}\nonumber\\
&\frac{\rho_{\tau}(a_{\v})}{|(\tau^{k+1})'(a_{\v})|}\left(\frac{1}{2|(\tau^{k+1})'(a_{-\v})|}\hat \rho_{\eps,\Lambda,k}(a_{-\v}^+)+\frac{1}{2|(\tau^{k+1})'(a_{-\v})|}\hat \rho_{\eps,\Lambda,k}(a_{-\v}^-)\right),\label{eq:limq_k}
\end{align}
with 
$$
\hat{\rho}_{\eps,\Lambda,k}=\E(1_{ (\Psi_k^{\q}=\p^*+\v') \cap \bigcap_{i=1}^r \,^c(\Psi_{j_i}^{\q}=\p^*+\w_{j_i})}\rho_{\eps,\Lambda}(.)|I^{\{q\}}).
$$
Using the above Lemma \ref{lemrefdenom} and \eqref{eq:simplified}, the quantity $\theta$ is then
$$
\theta=1-\sum_{\v,\v'\in V}\sum_{k\in \mathbb{K}(\v,\v')}q_{k}(\v,\v').
$$
\end{proof}
\begin{proof}[Proof of Theorem \ref{thm:main2}]
Using \eqref{eq:KLtheta} and Lemma \ref{lem:thetalemma} we obtain the first item of the lemma:
\begin{equation}\label{eq:asymp}
\lambda_{\varepsilon,\delta}=1-\mu_{\eps,\vp^*}(H_\delta(\vp^*))\cdot\theta(1+o(1)).
\end{equation} 
To prove the second item of the theorem, recall that the spectral data of $\LG_{\eps, \vp^*}$ do not depend on $\delta$ and 
$$\text{spec}(\LG_{\eps, \vp^*})\cap\{|z|>\gamma\}=\{1\},$$
where $\gamma\in(0,1)$ is as in \eqref{eq:contraction}. Choose, $r$ small enough, so that $\overline{B(1,r)}\cap \overline{B(0,\gamma)}=\emptyset$. Then, by the Keller-Liverani stability result \cite{KL99}, for $\delta$ small enough, $\frac{(1-r)-\gamma}{2}$ is a lower bound on the spectral gaps of $\LG_{\eps, \vp^*}$ and $\hat\LG_{\eps, \delta, \vp^*}$ for all $\delta\in[0,\delta^*]$ for small enough $\delta^*$. Therefore,
$$\kappa_N:=\sum_{k=N}^\infty\sup_{\delta\in[0,\delta^*]}\|\hat Q_{\eps,\delta, \vp^*}\|=\mathcal O\left(\left(\frac{1+r+\gamma}{2}\right)^N\right).$$
Recall
\begin{equation*} 
\eta_{\delta}=\sup_{\|\mu\|\le 1}\left|\left((\LG_{\eps,\p^*}-\hat\LG_{\eps,\delta,\p^*})\mu\right)(1)\right|,
\end{equation*} 
which was defined earlier in \eqref{eq:eta}. Let $N=\mathcal O\left(\log\eta_{\delta}\right)$. Then $N\eta_{\delta}+\mathcal O(\kappa_N)=\mathcal O\left(\eta_\delta\log\eta_{\delta}\right)$. Finally, let $\xi_{\delta}=\theta_{N,\delta}+\mathcal O(\kappa_N)$ where $\theta_{N,\delta}=1-\sum_{k=0}^{N-1} \lambda_{\eps,\delta}^{-k}q_{k,\delta}$.  Then the rest of the proof follows exactly as in \cite[Proposition 2]{K12}.
\end{proof}
\section{Proof of Theorem \ref{thm:main3}}\label{sec:proofthm3}
The proof is based on checking linear response of the following perturbed transfer operator
$$
\tilde \LG_{\delta,s}:=\LG_{\eps,\delta}(e^{is1_{H_\delta(\p^*)}}\cdot),
$$
where $s\in\mathbb R\setminus\{0\}$. To be precise about the definition of $\tilde{\mathcal{L}}_{\delta,s}(\cdot)$ we mean that for $\varphi\in\mathcal D_1$ and $\mu$ a Borel complex measure,
\begin{equation}\label{eq:optwist}
    \int \varphi d\tilde{\mathcal{L}}_{\delta,s}(\mu)=\int \varphi \circ T_{\eps,\delta} (e^{is1_{H_\delta(\p^*)}}) d\mu.
\end{equation}
Note that when $s=0$
$$\tilde \LG_{\delta,0}:=\LG_{\eps,\delta};$$ 
moreover, for any $s\in\mathbb R$ the iterates are given by $$\tilde{\LG}^n_{\delta,s}(\cdot)=\LG^n_{\eps,\delta}(e^{is\sum_{k=1}^n1_{H_\delta(\p^*)}\circ T_{\eps.\delta}^{k-1}}\cdot).$$
On the other hand, when $\delta=0$,
$$\tilde \LG_{0,s}=\LG_{\eps,\p^*}.$$
So we first show that for $\delta$ small enough $\tilde \LG_{\delta,s}$ and $\LG_{\eps,\p^*}$
satisfy assumptions (A1)-(A6) of \cite{KL09} in order to apply Theorem 2.1 of \cite{AFGV24}.
Note that by definition
$$|\tilde{\mathcal{L}}_{\delta,s}\mu|\le |\mu|;$$
Moreover, by Proposition \ref{propLasota-perturbed0} using in this case that $g_\delta=e^{is1_{H_\delta(\p^*)}}$, we obtain for $\eps>0$ small enough, there exist a $\sigma \in (0,1)$ and a $C>0$, independent of $\eps$, $\delta$ and $s$, such that for any $n\in \N$ and any $\mu \in \mathcal B$:
\begin{equation}\label{eq:LYtwist}
\|\tilde{\mathcal{L}}_{\delta,s}^n\mu\|\leq C\sigma^n\|\mu\|+C|\mu|.
\end{equation}

Note also that $\tilde{\mathcal{L}}_{\delta,s}$ can be decomposed as follows :
\begin{align}
\tilde{\mathcal{L}}_{\delta,s}(\cdot)&=(e^{is}-1)\mathcal{L}_{\eps,\delta}(1_{H_\delta(\p^*)}\cdot)+\mathcal{L}_{\eps,\delta}(\cdot)\label{eqpoissondecomp}\\
&=e^{is}\mathcal{L}_{\eps,\delta}(1_{H_\delta(\p^*)}\cdot)+\mathcal{L}_{\eps,\p^*}(1_{X_{0,\delta}(\p^*)}\cdot).\nonumber
\end{align}

We need to prove that $\tilde{\mathcal{L}}_{\delta,s}(\cdot)$ admits a spectral gap, for $\delta$ small enough. For this purpose we prove the following lemma.
\begin{lem}\label{lem:twistclose}
\begin{align}\label{eqlempoissonclose}
\sup_{\mu \in \B, \|\mu\|\leq 1}|\tilde{\mathcal{L}}_{\delta,s}(\mu)-\mathcal{L}_{\eps,\p^*}(\mu) |=\mathcal O({\delta})
\end{align}
and 
\begin{align}\label{eqlempoissonclose2}
\|\tilde{\mathcal{L}}_{\delta,s}(\mu_{\eps,\p^*})-\mathcal{L}_{\eps,\p^*}(\mu_{\eps,\p^*}) \|=\mathcal O({\delta}).
\end{align} 
\end{lem}
\begin{proof}
To prove \eqref{eqlempoissonclose} 
Let $\mu\in \mathcal B$ and $\varphi\in\mathcal D_1$. By \eqref{eqpoissondecomp}, we have
\begin{align*}
    \int\varphi d\left[\tilde{\mathcal{L}}_{\delta,s}(\mu)-\mathcal{L}_{\eps,\p^*}(\mu)\right]&=e^{is}\int \varphi \circ T_{\eps,\delta}1_{H_\delta(\p^*)}d\mu -\int \varphi\circ T_{\eps,\p^*} 1_{H_\delta(\p^*)}d\mu.
\end{align*}
By the same reasoning as in \eqref{eqlemclose}, each of the above terms is $\mathcal O(\delta)$ above. Thus,
\begin{align*}
\sup_{\mu \in \B, \|\mu\|\leq 1}|\tilde{\mathcal{L}}_{\delta,s}(\mu)-\mathcal{L}_{\eps,\p^*}(\mu) |=\mathcal O({\delta}).
\end{align*}
To prove \eqref{eqlempoissonclose2} we fix $\phi \in \mathcal D$ a local map depending of coordinate in a finite set $\Lambda_0\subset \ZZ^d$ and $r\in \ZZ^d$,
\begin{align}\label{eqkellera3}
    \left[\tilde{\mathcal{L}}_{\delta,s}(\mu_{\eps,\p^*})-\mathcal{L}_{\eps,\p^*}(\mu_{\eps,\p^*})\right](\partial_{r}\phi)&=e^{is}\int \partial_{\q}\phi\circ T_{\eps,\delta} 1_{H_\delta(\p^*)}d\mu_{\eps,\p^*}\\
    &\hskip 0.5cm -\int \partial_{r}\phi\circ T_{\eps,\p^*} 1_{H_\delta(\p^*)}d\mu_{\eps, \p^*}.\nonumber
\end{align}
In the above expression, the second integral is bounded in Lemma \ref{lem:close}. 
 For the other integral, we introduce the following notation. We choose an element $(\q,\v)\in \mathcal{N}_{\p^*}$ and let $\Phi_{\q}^\tau$ be given by
\begin{equation}\label{coupling1}
(\Phi_{\q}^\tau(\vx))_{\p}=\begin{cases}
       x_{\q} \quad \text{if } x_{\p}\in \tau A_{\delta,\v}\text{ and } x_{\q}\in \tau A_{\delta,-\v} \text{ for some } \v\in V \text{ and } \p=\p^*\\
       x_{\p^*} \quad \text{if } x_\p\in \tau A_{\delta,\v}\text{ and } x_{\q}\in \tau A_{\delta,-\v} \text{ for some } \v\in V \text{ and }\p=\q\\
       x_\p  \quad\quad\text{otherwise.}
       \end{cases}
\end{equation}
Then one can split the first integral in \eqref{eqkellera3} into $|\mathcal{N}_{\p^*}|$ integrals of the following form:
\begin{align*}
    &\int \partial_{r}\phi\circ T_{\eps,\delta} 1_{A_{\delta,\v}}\circ\pi_{\p^*}1_{A_{\delta,-\v}}\circ\pi_{\q}d\mu_{\eps,\p^*}\\
    &=\int \partial_{r}\phi\circ \Phi_q^\tau\circ T_{\eps,\p^*} 1_{A_{\delta,\v}}\circ\pi_{\p^*}1_{A_{\delta,-\v}}\circ\pi_{\q}d\mu_{\eps,\p^*}
\end{align*}
for each $(\q,\v)\in \mathcal{N}_{\p^*}$. Let $\vx\in I^{\ZZ^d}$ we choose to represent such element as $\vx:=(\vx_1,\vx_2,\vx_3)\in I^{\{\p^*\}}\times I^{\{\q\}}\times I^{\ZZ^d\backslash \{\p^*,\q\}}$ 
notice then, that for $\Lambda:=\Lambda_0\cup \{\p+\v, \p\in \Lambda_0, \v \in V\}$ (similarly as in the proof of Lemma \ref{lem:verirare}) $\tilde \phi \circ\pi_{\Lambda} (\vx_\p,\vx_\q,\vx_{\neq}):= \phi(\vx_\q,\vx_\p,\vx_{\neq})$ satisfies $\partial_{r}\phi\circ \Phi_{\q}^\tau\circ T_{\eps,\p^*}(\vx)=\partial_{r}\tilde \phi\circ\pi_{\Lambda_0}\circ T_{\eps,\p^*}(\vx)$ for any $\vx\in I^{\ZZ^d}$ and $\tilde \phi \in \D_1$ with $\|\phi\|_\infty=\|\tilde \phi\|_\infty$. Thus, one can follow the same reasoning as in the second point of Lemma \ref{lem:verirare} with $\tilde \phi$ instead of $\varphi$. We then reach the same conclusion as equation \eqref{eqpartialphibv} : 

\begin{align*}
    &\left|\int \partial_r \tilde \phi\circ \pi_{\Lambda}\circ \tilde T_{\eps, \p^*} 1_{A_{\delta,-\v}}\circ \pi_{\q} \rho_{\eps,\Lambda_0\backslash \{\p^*\}}dm_{\Lambda_0\backslash \{\p^*\}}  1_{A_{\delta,\v}}\circ \pi_{\p^*}\rho_{\eps,\p^*}dm\right|\\
&\leq \|\mu_{\eps,\p^*}\|^2\delta.
\end{align*}

which brings us to the conclusion of the second statement
\begin{align*}
\|\tilde{\mathcal{L}}_{\delta,s}(\mu_{\eps,\p^*})-\mathcal{L}_{\eps,\p^*}(\mu_{\eps,\p^*}) \|=\mathcal O({\delta}).
\end{align*} 
\end{proof}
The uniform Lasota-Yorke inequality \eqref{eq:LYtwist}, together with \eqref{eqlempoissonclose} of Lemma \ref{lem:twistclose}, for $\delta$ sufficiently small, the Keller-Liverani stability result \cite{KL99} implies that the operator $\LG_{\delta,s}$ admits a spectral gap on $\mathcal B$. This proves assumptions (A1)-(A4) of \cite{KL09}.
Now let
\begin{align*}
\tilde{\eta}_\delta &=\sup_{\|\mu\|\leq 1}|\left(\LG_{\eps,\p^*}(\mu)-\tilde\LG_{s,\delta}(\mu)\right)(1)|;\\
\tilde \Delta_\delta &=\left(\LG_{\eps,\p^*}(\mu_{\eps,\p^*})-\tilde\LG_{s,\delta}(\mu_{\eps,\p^*})\right)(1).
\end{align*}
By Lemma \ref{lem:twistclose} and Lemma \ref{lemrefdenom} we have
\begin{cor}\label{cor:verirare}
The following holds   
\begin{enumerate}
\item $\tilde\eta_\delta\to 0$ as $\delta\to 0$;
\item $\exists\, C>0$ such that $\tilde\eta_\delta\cdot\|(\LG_{\eps,\p^*}-\tilde\LG_{s,\delta})\mu_{\eps,\p^*}\|\le C |\tilde\Delta_\delta|$.
\end{enumerate}
\end{cor}
Thus, by Corollary \ref{cor:verirare} conditions (A5)-(A6) of \cite{KL09} are satisfied. Moreover, notice that
\begin{align*}
\tilde{\eta}_\delta &=|e^{is}-1|\eta_\delta\\
\tilde \Delta_\delta(t)
&=(1-e^{is})\Delta_\delta=(1-e^{is})\mu_{\eps,\p^*}(H_\delta(\p^*)),
\end{align*}
where $\eta_\delta$ and $\Delta_\delta$ are defined in equations \eqref{eq:eta} and \eqref{eq:Delta} respectively.
\subsection{Existence of $\tilde\theta$}
We are left to check the existence of the following limit when $\delta \to 0$:
\begin{align}\label{eqpoissonqk}
&\tilde q_{k,\delta}(s):= \frac{((\LG_{\eps,\p^*}-\tilde{\LG}_{\delta,s})\tilde{\mathcal{L}}_{\delta,s}^k((\LG_{\eps,\p^*}-\tilde{\LG}_{\delta,s})\mu_{\eps,\p^*})(1)}{\tilde \Delta_\delta(s)}\nonumber\\
&=(1-e^{is})\frac{\int 1_{H_\delta(\p^*)}d\tilde{\mathcal{L}}_{\delta,s}^k((\LG_{\eps,\p^*}(1_{H_\delta(\p^*)}\cdot)-e^{is}\LG_{\eps,\delta}(1_{H_\delta(\p^*)}\cdot)\mu_{\eps,\p^*})}{\tilde \Delta_\delta(s)}\nonumber\\
&=\frac{\int 1_{H_\delta(\p^*)}\circ T_{\eps,\delta}^k e^{is\sum_{j=0}^{k-1}1_{H_\delta(\p^*)}\circ T_{\eps,\delta}^j}d((\LG_{\eps,\p^*}(1_{H_\delta(\p^*)}\cdot)-e^{is}\LG_{\eps,\delta}(1_{H_\delta(\p^*)}\cdot)\mu_{\eps,\p^*})}{\Delta_\delta}\\
&=\frac{1}{\Delta_\delta}\sum_{j=0}^ke^{isj}\mu_{\eps,\p^*}\left(T_{\eps,\p^*}^{-1}\left(T_{\eps,\delta}^{-k}H_\delta(\p^*)\cap \left(\sum_{r=0}^{k-1}1_{H_\delta(\p^*)}\circ T_{\eps,\delta}^r=j\right)\right)\cap H_{\delta}(\p^*)\right)\nonumber\\
&-\frac{1}{\Delta_\delta}\sum_{j=0}^ke^{is(j+1)}\mu_{\eps,\p^*}\left(T_\eps^{-k-1}H_\delta(\p^*)\cap \left(\sum_{r=1}^{k}1_{H_\delta(\p^*)}\circ T_{\eps,\delta}^r=j\right)\cap H_{\delta}(\p^*)\right)\nonumber\\
&=\sum_{j=0}^ke^{isj}\left(\beta_{k,\delta}^{(1)}(j)-e^{is}\beta_{k,\delta}^{(2)}(j)\right),
\end{align}
with 
$$
\beta_{k,\delta}^{(1)}(j):=\frac{\mu_{\eps,\vp^*}\left(H_\delta(\vp^*)\cap T_{\eps,\p^*}^{-1}\left(T_{\eps,\delta}^{-k}H_\delta(\vp^*)\cap \left(\sum_{i=0}^{k-1}1_{H_\delta(\p^*)}\circ T_{\eps,\delta}^i=j\right)\right)\right)}{\mu_{\eps,\vp^*}\left(H_\delta(\vp^*)\right)}.
$$
and
$$
\beta_{k,\delta}^{(2)}(j):=\frac{\mu_{\eps,\vp^*}\left(H_\delta(\vp^*)\cap T_{\eps,\delta}^{-(k+1)}H_\delta(\vp^*)\cap \left(\sum_{i=1}^k1_{H_\delta(\p^*)}\circ T_{\eps,\delta}^i=j\right)\right)}{\mu_{\eps,\vp^*}\left(H_\delta(\vp^*)\right)}.
$$ 

In what follows, we just treat the existence of a limit for $\beta_{k,\delta}^{(2)}(j)$, the proof for $\beta_{k,\delta}^{(1)}(j)$ follows the same path.
We introduce the maps $\tilde \Psi_k^{\q}$ as follows\footnote{For $\beta_{k,\delta}^{(2)}(j)$ one may use $\hat \Psi_k^{\vp}$ with $\hat \Psi_1^{\vp}=\Psi_1^{\vp}$ and $\hat \Psi_k^{\vp}(\vx)=\tilde \Psi_{k-1}^{\vp}(\vx)+\tilde \Psi_1^{\tilde \Psi_{k-1}^{\vp}(\vx)}(T_{\eps,\delta}^{k-2}\circ T_{\eps,\p^*})(\vx)$.} :
\begin{align}\label{formulpsidelta}
\left\{\begin{array}{r l c}
 \tilde \Psi_1^{\vp}(\vx)&=&\p+\v \text{ if } x_\p\in A_{\eps,\v}\text{ and } x_{\p+\v}\in A_{\eps,-\v} \text{ for some } \v\in V\\ 
\tilde \Psi_1^{\vp}(\vx)&=& \p \text{ otherwise}\\
\tilde \Psi_k^{\vp}(\vx)&=&\tilde \Psi_{k-1}^{\vp}(\vx)+\tilde \Psi_1^{\tilde \Psi_{k-1}^{\vp}(\vx)}(T_{\eps,\delta}^{k-1}\vx) \text{ if }k\neq 0.
\end{array}
\right.
\end{align}
Notice that $\tilde \Psi_k^{\vp}$ depends on $\delta$.

Proving the existence of $\lim_{\delta \to 0}  \beta_{k,\delta}^{(i)}(j,\v,\v')$ follows the same steps as what we have done for the existence of $\lim_{\delta \to 0} q_{k,\delta}$ in the proof of Lemma \ref{lem:thetalemma}. In other words we split $\beta_{k,\delta}^{(i)}(j)$ into the counterpart of \eqref{eq:tildeqvv}. 

\begin{align}
    \beta_{k,\delta}^{(i)}(j):=\sum_{\v,\v'\in V^{+}} \beta_{k,\delta}^{(i)}(j,\v,\v'),
\end{align}
where 
\begin{equation}
  \beta_{k,\delta}^{(i)}(j,\v,\v')= \frac{1}{\mu_{\eps,\p^*}(H_\delta(\p^*))}\mu_{\eps,\p^*}\left(\tilde G_{\v,\v'}\right),
\end{equation}
with 
$\tilde G_{\v,\v'}$ being the short notation for
\begin{align}
A_{\delta,\v}(\p^*)\cap A_{\delta,-\v}(\q) &\cap T_{\eps,\delta}^{-(k+1)}\left(A_{\delta,\v'}(\p^*)\cap A_{\delta,-\v'}(\q')\right)\nonumber\\ 
&\cap \left(\sum_{i=1}^k1_{H_\delta(\p^*)}\circ T_{\eps,\delta}^i=j\right). 
\end{align}
To prove the existence of $\lim_{\delta \to 0}  \beta_{k,\delta}(j,\v,\v')$ we will actually prove that for $\delta$ small enough $\E_{\mu_{\eps,\p^*}}\left(1_{\tilde G_{\v,\v'}}|I^{\p^*,\p^*+\v}\right)$ can be decomposed into
\begin{align}
    \E_{\mu_{\eps,\p^*}}\left(1_{\tilde G_{\v,\v'}}|I^{\p^*,\p^*+\v}\right)&=1_{A_{\delta,\v}(\p^*)}1_{\tau^{-(k+1)}A_{\delta,\v'}(\p^*)}\tilde \rho_{\eps,\p^*}\circ \pi_{\p^*}\\
    &\hskip 0.5cm 1_{A_{\delta,-\v}(\p^*+\v)}1_{\tau^{-(k+1)}A_{\delta,-\v'}(\p^*+\v)}\tilde \rho_{\eps,\p^*+\v}\circ\pi_{\p^*+\v}\nonumber\\
    &+o(\delta^2)\nonumber,
\end{align}
where $\pi_{\p^*}$ and $\tilde \rho_{\eps,\p^*+\v}$ are functions of bounded variation and they are independent of $\delta$. Once this done one can consider the limit $\lim_{\delta \to 0}  \beta_{k,\delta}(j,\v,\v')$ and obtain a formula similar to \eqref{eq:limq_k}.\\

Now notice that from the definition of the set $J_k$ associated to a couple $(a_{\v},a_{-\v})\in S^{rec}$, for $x \in A_{\delta,\v}(\p^*)\cap A_{\delta,-\v}(\q)$ the sum $\sum_{i=1}^k1_{H_\delta(\p^*)}\circ T_{\eps,\delta}^i(x)$ may be reduced to $\sum_{i=1}^r1_{H_\delta(\p^*)}\circ T_{\eps,\delta}^{j_i}(x)$ for $\delta$ small enough. 
Actually one can notice that Lemma \ref{lemnegnonrec} extends to that setting with maps $\tilde \Psi_{j_i}^{\q}(\vx)$ into the following lemma:
\begin{lem}\label{lemnegnonrec2}
For  any $(\q,\v) \in \mathcal N_{\vp^*}$, $k>0$ and any $(\q',\v') \in \mathcal N_{\vp^*}$
$$
\lim_{\delta \to 0}\frac{1}{\mu_{\eps,\p^*}(A_{\delta,\v}(\q))}\mu_{\eps,\p^*}\left(A_{\delta,\v}(\q)\cap T_{\eps,\vp^*}^{-k}A_{\delta,\v'}(\q')\cap (\tilde \Psi_k^{\q}\neq \q')\right)=0. 
$$
In addition and for any $j\leq k$,
\begin{align}
&\lim_{\delta \to 0}\frac{1}{\mu_{\eps,\p^*}(A_{\delta,\v}(\q))}\mu_{\eps,\p^*}\left(A_{\delta,\v}(\q)\cap T_{\eps,\vp^*}^{-k}A_{\delta,\v}(\q)\cap (\tilde \Psi_j^{\q}\neq \p^*)\cap (\tilde \Psi_j^{\p^*}\neq \p^*)\right)\nonumber\\
&=0.
\end{align}
\end{lem}
The proof of the above lemma follows the exact same reasoning as Lemma \ref{lemnegnonrec} with $\tilde \Psi_j^{\p^*}$ instead of $\Psi_j^{\q}$ (and $T_{\eps,\delta}$ instead of $T_{\eps,\p^*}$).\\
So for $\delta$ small enough, the values taken by the sum $\sum_{i=1}^r1_{H_\delta(\p^*)}\circ T_{\eps,\delta}^{j_i}(x)$ (still for $x \in A_{\delta,\v}(\p^*)\cap A_{\delta,-\v}(\q)$) are entirely determined by the values taken by the map $(\tilde \Psi_{j_i}^{\p}(\vx),\tilde \Psi_{j_i}^{\q}(\vx))$ which index the site impacted at time $j_i$ by any small variation of the state of $\vx$ in site $\q$ and $\p$ at time $0$.
When 
 \begin{align*}
\sum_{i=1}^k1_{H_\delta(\p^*)}\circ T_{\eps,\delta}^i(x)=\sum_{i=1}^k1_{\{\tilde \Psi_{j_i}^{\p^*}(\vx),\tilde \Psi_{j_i}^{\q}\}=\{\p^*,\p^*+\w_{j_i}\}}.
\end{align*}
Therefore, one can divide the set $\left\{\sum_{i=1}^r1_{H_\delta(\p^*)}\circ T_{\eps,\delta}^{j_i}(x)=m \right\}$ into subsets of the form $W_\delta:=\bigcup_{P\in P_m\{j_1,\dots,j_r\}} \bigcap_{p_i\in P} T_{\eps,\delta}^{-p_i}H_{\delta}(\p^*)\bigcap_{p_i\notin P}T_{\eps,\delta}^{-p_i}\left(X_{0,\delta}(\p^*)\right)$ and expand it using Lemma \ref{lemnegnonrec2} and the fact that $\tau$ is a uniformly expanding map, we deduce that
\begin{align}
    &\lim_{\delta \to 0}\frac{1}{\mu_{\eps,\p^*}(H_\delta(\p^*))}\mu_{\eps,\p^*}\left(A_{\delta,\v}(\p^*)\cap T_{\eps,\delta}^{-(k+1)}A_{\delta,\v'}(\p^*)\cap W_\delta\right)\label{eqrefpoissimp1}\\
&=\lim_{\delta \to 0}\frac{1}{\mu_{\eps,\p^*}(H_\delta(\p^*))}\nonumber\\
&\mu_{\eps,\p^*}\left(A_{\delta,\v}(\p^*)\cap A_{\delta,-\v}(\p^*+\v)\cap T_0^{-(k+1)}A_{\delta,\v'}(\p^*)\cap T_0^{-(k+1)}A_{\delta,-\v'}(\p^*+\v)\right.\label{eqrefpoissimp2}\\
&\left.\hskip 1cm\bigcup_{P\in P_m\{j_1,\dots,j_r\}}\bigcap_{p_i\in P} (\tilde{\Psi}_{j_i}^{\q},\tilde{\Psi}_{j_i}^{\p^*})\in\{\p^*,\p^*+\w_{j_i}\}\right.\nonumber\\
&\left.\hskip 2cm \cap \bigcap_{p_i\notin P}\left((\tilde{\Psi}_{j_i}^{\q},\tilde{\Psi}_{j_i}^{\p^*})\notin\{\p^*,\p^*+\w_{j_i}\}\right)\right)\nonumber\\
&=\lim_{\delta \to 0}\frac{1}{\mu_{\eps,\p^*}(H_\delta(\p^*))}\nonumber\\
&\sum_{P\in P_m\{j_1,\dots,j_r\}}\mu_{\eps,\p^*}\left(A_{\delta,\v}(\p^*)\cap T_0^{-(k+1)}A_{\delta,\v'}(\p^*)\cap D_P\right)\label{eqrefpoissimp3}
\end{align}
with 
\begin{align*}
     D_P:=&\bigcap_{p_i\in P} \left( (\tilde{\Psi}_{j_{p_i}-j_{p_{i-1}}-1}^{\q}\left(\tilde T_{p_{i-1}+1}(\vx)\right)=\p^*+\w_{p_i})\right)\\
     &\cap \bigcap_{p\in \{j_1,\dots,j_r\}\backslash P}\left((\tilde{\Psi}_{j_p-j_{p_i}-1}^{\q}\left(\tilde T_{p_i+1}(\vx)\right)\neq\p^*+\w_{p_i})\right),
\end{align*}
where the set $P\in P_m\{j_1,\dots,j_r\}$ at the last line of \eqref{eqrefpoissimp3} is ordered so that $P=\{p_1,\dots,p_{|P|}\}$ with $p_1\leq \dots \leq p_{|P|}$ and for $p\in\{1,\dots,r\}$, $p_i$ is the largest element of $P$ so that $p>p_i$ and we define the map $\tilde T_{p_i+1}$ as
$$\tilde T_{p_i+1}:= \Phi\circ T_{\eps,\p^*}^{p_i-p_{i-1}}\circ \dots \circ \Phi \circ T_{\eps,\p^*}^{p_2-p_{1}}\circ  \Phi \circ T_{\eps,\p^*}^{p_1}\circ \Phi\circ T_{\eps,\p^*}.
$$
From equation \eqref{eqrefpoissimp3} we can deduce that whenever $\tilde S^{rec}=\emptyset$, then\footnote{Similarly, when $S^{rec}=\emptyset$, $\lim_{\delta \to 0} \beta_{k,\delta}^{(1)}(j)=0$.} 
$$
\lim_{\delta \to 0} \beta_{k,\delta}^{(2)}(j)=0.
$$
Notice in addition that on $\bigcap_{p\in P} T_{\eps,\delta}^{-p}H_{\delta}(\p^*)\bigcap_{p\notin P,p\leq k}T_{\eps,\delta}^{-p}\left(X_{0,\delta}(\p^*)\right)$, all the interactions with the site $\p^*$ occur at precise identified time, and from Lemma \ref{lemnegnonrec2}, we always have $\p^*\in \{\tilde{\Psi}_{j}^{\q},\tilde{\Psi}_{j}^{\p^*}\}$, so we only need to keep trace of one of those two maps at a time, this is what is displayed at the last line of equation \eqref{eqrefpoissimp3} with the map $\tilde{\Psi}_{p_i-p_{i-1}}^{\q}$.
The point here is that on such a set, this map $\tilde{\Psi}_{j_p-j_{p_i}-1}^{\q}\left(\tilde T_{p_i+1}(\vx)\right)$ actually does not depend on $\delta$ : $\tilde{\Psi}_{j_p-j_{p_i}-1}^{\q}\left(\tilde T_{p_i+1}(\vx)\right)=\Psi_{j_p-j_{p_i}-1}^{\q}\left(\tilde T_{p_i+1}(\vx)\right)$.
Therefore, for $P\in \mathcal{P}\{1,\dots,r\}$, the map $\tilde \rho:=\rho_{\eps}\E_{\mu_{\eps,\p^*}}\left(1_{D_P}|I^{\{\p^*,\q\}}\right)$ is independent of $\delta$ and our last step is to prove that it can be decomposed on $I^{\{\p,\q\}}$ into a linear combination of product $\rho_{\p^*}(x_{\p^*})\rho_{\q}(x_{\q})$ of one dimensional functions of bounded variation. To do this, we first introduce the following notion:\\

\begin{defn}\label{defncompoundcylindfunct}
We call cylinder map any finite linear combination of maps of the form $1_{A_1\times \dots \times A_{|\Lambda|}}\circ \pi_{\Lambda}$ where 
$\Lambda \subset \ZZ^d$ is a finite set and $A_1,\dots, A_{|\Lambda|}$ are a set of sub-intervals of $I^{\q}$ for each $\q\in \Lambda$.\\
\end{defn}
Notice that if $f$ is a cylinder map, then for any $p\in \mathbb{R}$,
$1_{f=p}$, $f\circ T_0$, $f\circ \Phi$, $\tilde \Psi$ and $f\circ \Phi_\eps$ are cylinder map. And for $\Lambda\subset \ZZ^d$, $\E_{\mu_{\eps,\p^*}}\left(f| I^{\Lambda}\right)$ is a cylinder map on $I^{\Lambda}$. Therefore, $\E_{\mu_{\eps,\p^*}}\left(1_{D_P}|I^{\{\p^*,\q\}}\right):=\sum_k 1_{R^P_k}(\vx_{\p^*})1_{B^P_k}(\vx_{\q})$ is a cylinder map and
\begin{align}
    &\E_{\mu_{\eps,\p^*}}\left(1_{\tilde G_{\v,\v'}}|I^{\p^*,\p^*+\v}\right)\\
    =&\sum_{P\in \mathcal{P}\{1,\dots,r\} }\sum_{k} \int 1_{A_{\delta,\v}}(x)1_{\tau^{-(k+1)}A_{\delta,\v}}(x)\rho_\tau(x)1_{R^P_k}(x)dx\\
    &\int 1_{A_{\delta,\v}}(y)1_{\tau^{-(k+1)}A_{\delta,\v}}(y)\rho_q(y)1_{B^P_k}(y)dy.
\end{align}
Since $\rho_q(\cdot)1_{B^P_k}(\cdot)$ corresponds to some piecewise continuous map, through the same reasoning as in \eqref{eqestimatneighbor}, the following limit exists  
\begin{align*}
    \lim_{\delta \to 0}\frac{1}{m(A_{\delta,\v}\times A_{\delta,-\v})}&\int 1_{A_{\delta,\v}}(x)1_{\tau^{-(k+1)}A_{\delta,\v}}(x)\rho_\tau(x)1_{R^P_k}(x)dx\\
    &\int 1_{A_{\delta,\v}}(y)1_{\tau^{-(k+1)}A_{\delta,\v}}(y)\rho_q(y)1_{B^P_k}(y)dy.
\end{align*}
Consequently, for $i\in \{1,2\}$, the following limit exists 
\begin{align*}
\beta_k^{(i)}(j,\v,\v'):=\lim_{\delta \to 0}\beta_{k,\delta}^{(i)}(j,\v,\v').
\end{align*}
Using the above Lemma \ref{lemrefdenom} and \eqref{eq:simplified}, we deduce that
\begin{align}
    \lim_{\delta \to 0}\beta_{k,\delta}^{(i)}(j)=\sum_{\v,\v'\in V^{+}} \beta_k^{(i)}(j,\v,\v'),
\end{align}
and the quantity $\tilde\theta$ is then given by
\begin{align*}
\tilde \theta(s)=&1-\sum_{k\geq 1}\tilde q_{k,\delta}(s) \\
=&1-\sum_{\v,\v'\in V}\sum_{k\in \mathbb{K}(\v,\v')}\sum_{j=0}^ke^{isj}\left(\beta_k^{(1)}(j,\v,\v')-e^{is}\beta_k^{(2)}(j,\v,\v')\right).\\
\end{align*}
\begin{proof}[Proof of Theorem \ref{thm:main3}]
By the above, we have shown that the dominant eigenvalue, $\lambda_{s,\delta}$, of $\tilde\LG_{s,\delta}$ has the following first order approximation
\begin{align*}
1-\lambda_{s,\delta}&=\tilde \Delta_\delta\tilde\theta(s)(1+o(1))\\
&=(1-e^{is})\mu_{\eps,\p^*}(H_{\delta}(\p^*))\tilde\theta(s)(1+o(1)),    
\end{align*}
that is
\begin{equation}\label{eq:exponantiating}
\lambda_{s,\delta}=e^{-\mu_{\eps,\p^*}(H_{\delta}(\p^*))(1-e^{is})\tilde\theta(s)(1+o(1))}.
\end{equation}
Recall from \eqref{eq:counting} the definition of the process 
$$Z_\delta(t)=\sum_{k=1}^{\left\lfloor\frac{t}{\mu_{\eps,\p^*}(H_\delta(\p^*))}\right\rfloor}1_{H_\delta(\p^*)}\circ T_{\eps,\delta}^k$$ and the fact that
\begin{equation}\label{eq:twistit}
\tilde{\LG}^{\left\lfloor\frac{t}{\mu_{\eps,\p^*}(H_\delta(\p^*))}\right\rfloor}_{\delta,s}(\cdot)=\LG^{\left\lfloor \frac{t}{\mu_{\eps,\p^*}(H_\delta(\p^*))}\right\rfloor}_{\eps,\delta}(e^{isZ_\delta(t)}\cdot).
\end{equation}
Therefore, using the spectral decomposition of $\tilde{\LG}$, \eqref{eq:exponantiating} and \eqref{eq:twistit} we obtain for any $\mu\in\B$ 
$$\lim_{\delta\to 0}\int e^{isZ_\delta(t)}d\mu=e^{-(1-e^{is})\tilde\theta(s)t}.$$
Thus, using the L\'evy Continuity Theorem we can argue as in the discussion after Theorem 2.1 of \cite{AFGV24} to conclude that the process $Z_\delta(t)$ converges in law to a compound Poisson process $Z(t)=\sum_{i=1}^{N(t)}X_i$, where $(N(t))_t$ is a Poisson process of intensity
$\theta Leb_{|[0,\infty)}$ and $(X_i)_{i\in \N}$ is an iid sequence whose characteristic function is given by $\phi_X(s)=\frac{(1-e^{is})\tilde \theta(s)}{\theta}+1$.
\end{proof}
\subsection{Justification of \eqref{eq:examplepoisson} in the example}\label{subsect:just}
Here we justify the formula in \eqref{eq:examplepoisson} in the example of subsection \ref{sub:example}. To compute the limit law of the compound Poisson process $Z(t):=\sum_{i=1}^{N(t)}X_i$ from Theorem \ref{thm:main3}, one has to compute the limit of  
$\lim_{\delta\to 0}\beta_{k,\delta}^{(i)}(j)$ for $i\in \{1,2\}$ as presented under equation \eqref{eqpoissonqk}. Since $\tilde S^{rec}=\emptyset$, we deduce from equation \eqref{eqrefpoissimp3} that $\lim_{\delta \to 0}\beta_{k,\delta}^{(2)}(j,\cdot,\cdot)=0$ for all $j\in \NN$. Although $S^{rec}\neq 0$, one can also notice that $\lim_{\delta\to 0}\beta_{k,\delta}^{(1)}(j)=0$ for $j\geq 1$.

To see it, fix $k\geq 0$ and choose any $\delta$ small enough so that $\tau^{m}x$ remains close to $a_{\v}$ for any $m\leq k+1$ and $\v\in V$ assume there is an intermediate collision with the site $\p^*$ at time $r\leq k$ from the neighboring site  $\q=\p^*+\v$. Then at time $r+m\geq r+1$ for a realisation $\vx$ of the probability in play in  
$\lim_{\delta\to 0}\beta_{k,\delta}^{(1)}(j)$, $\left(T_{\epsilon,\delta}^{m+r}(\vx)\right)_{\q}$ remains close to $\tau^{m}a_{\q+\v}=a_{\q+\v}$ for any $m$ such that $\tilde \Psi_k^{\vp^*}(\vx)=\q$ with $\tilde \Psi_k^{\vp^*}(\vx)$ given by the formula \eqref{formulpsidelta} and playing a similar role for $\beta_{k,\delta}^{(1)}(j)$ as $\Psi_k^{\q}$ for $q_{k,\delta}$. Thus there is no collision for such values of $m$. For the other values of $m\leq k+1-r$,  Lemma \ref{lemnegnonrec2} ensures that the relative weight of such collision would be negligible as $\delta\to 0$.
We thus can conclude that 
$$
\tilde \theta(s)=1-\sum_{\v,\v'\in V}\sum_{k\in {\KK(\v,\v')}}\beta_{k,\delta}^{(1)}(j,\v,\v')=\theta
$$
and that the law of the integer variable $X_i$ is given by the following characteristic function $\phi_{X}(s)=e^{is}$. In other words, $X_i$ is almost surely constant and equal to $1$ and $Z(t)\sim \mathcal P(\theta t)$, that is a Poisson law of intensity $\theta t$ with $\theta=1-5^{-4}$.
\section{An Auxiliary Uniform Lasota-Yorke inequality}\label{sec:LYaux}
In this section we prove an auxiliary Lasota-Yorke inequality which can be applied to different transfer operators in sections \ref{sec:pf1}, \ref{sec:pf2} and \ref{sec:proofthm3}.
For $\varphi\in\mathcal D$ and $\mu$ be a Borel complex measure, recall
$$   \int \varphi d\LG_{\eps,\delta}\mu=\int \varphi \circ T_{\eps,\delta} d\mu,
$$
where $T_{\eps,\delta}$ is the fully coupled map lattice defined in \eqref{eq:fullymap}, with $\delta\in(0,\eps]$ being the size of collision zones involved with site $\p^*$.
\begin{prop}\label{propLasota-perturbed0}
Let $\Lambda \subset \ZZ^d$ be a finite subset and $(A_i)_{i=1,\dots,s} \subset  (I^{\Lambda})^s$ be a finite family of sets with $A_i$ being a product of subintervals of $I$ whose Lebesgue measure depends on $\delta>0$. Let $a_i\in\mathbb R$, $|a_i|\le 1$, $i=1,\dots,s$ with $a_1\in \{0,1\}$. Define
$h_\delta=\sum_{i=1}^s a_i1_{A_i}$ and assume that whenever $\delta \geq \delta'$, $\{h_\delta=a_1\}\subset \{h_{\delta'}=a_1\}$ and  $\{h_{\delta'}=a_i\}\subset \{h_{\delta}=a_i\}$ for all other $i>1$. Let $g_\delta: I^{\mathbb Z^d}\to\mathbb R$ be given by $g_\delta(\textnormal{\vx}):=h_\delta(\textnormal{\vx}_{\Lambda})$ and consider the following operator
$$
\LG_{\eps,g_\delta}:=\LG_{\eps,\delta}(g_\delta\cdot). 
$$
Then $\exists$ $\sigma \in (0,1)$ and $C>0$, independent of $\eps$ and $\delta$, such that for any $n\in \N$ and any $\mu \in \mathcal B$:
\begin{align*}
\|\LG_{\eps,g_\delta}^n\mu\|\leq C\sigma^n\|\mu\|+C|\mu|.
\end{align*}
\end{prop}
\begin{proof}
Notice that 
\begin{align}\label{eq:auxop}
    \LG_{\eps,g_\delta}^n(\cdot)=\LG_{\eps,\delta}^n\left(\prod_{k=0}^{n-1}g_\delta \circ T^k_{\eps,\delta}\right),
\end{align}
and that if $g_\delta$ can take at most $r$ different values then the map $\prod_{k=0}^{n-1}g_\delta \circ T^k_{\eps,\delta}$ 
can take at most $n^r$ different values that we denote $a_1,\dots,a_{m_r}$ with $m_r\leq n^r$.

We first prove the following statement. Let $b\in \{0,1\}$, there is $0<\sigma_0<\frac12$, $n^*\in \mathbb N$ and $C_{n^*}>0$ such that for any $a_i \in \text{Image}(g_{\delta}^n)$, 
\begin{align}\label{eqlymauv}
    \|\LG_{\eps, \delta}^{n^*}\left(1_{B_{\delta,n^*,a_i}^b}\mu \right)\|\leq 2\sigma_0\|\mu\|+C_{n^*}|\mu|,
\end{align}
with $B_{\delta,n,a_i}^0:=\{g_{\delta}^n=a_i\}$ and $B_{\delta,n,a_i}^1:=^c\{g_{\delta}^n=a_i\}$
and notice that
\begin{align}\label{eq:auxop}
    \LG_{\eps,g_\delta}^n(\cdot)=\sum_{i=1}^{m_r}a_i\LG_{\eps,\delta}^n\left(1_{B_{\delta,n,a_i}^0}\cdot\right).
\end{align}
Let
$$\LG_{\delta,n,a_i,0}':=\LG_{\eps, \delta}^n\left(1_{B_{\delta,n,a_i}^0}\cdot\right) \text{ and }\LG_{\delta,n,a_i,1}':=\LG_{\eps, \delta}^n\left(1_{^c\{g_{\delta}^n=a_i\}}\cdot\right).$$
We start by proving \eqref{eqlymauv} for the operator $\LG_{\delta,n,a_i,0}'$ for $i\neq 1$. We will eventually prove it for any index $i$. 
We introduce
$$\Delta_{\eps,\q,\v}:=A_{\eps,\v}(\q)\cap A_{\eps,-\v}(\q+\v)$$

and 
$$\Delta_{\eps,\vq,0}:=\, ^c\left(\bigcup_{\v\in V}\Delta_{\eps,\q,\v}\right).$$
Let $\varphi \in \mathcal D_1$. If $\vx \in \Delta_{\eps,\vq,0}$, then
\begin{align*}
\partial_{\q} \left(\frac{\varphi\circ T_{\eps,\delta}}{\tau'\circ \pi_\q}\right)=\partial_{\q} \left(\varphi\circ T_{\eps,\delta} \right)\frac 1 {\tau'\circ \pi_\q}- \left(\varphi\circ T_{\eps,\delta} \right)\frac {\tau''\circ\pi_\q} {(\tau')^2\circ \pi_\q}.
\end{align*}
Note that we use the assumption that the site dynamics $\tau$ is piecewise onto. On the other hand if $\vx\in \Delta_{\eps,\q,\v}$ then
\begin{align*}
\partial_{\q+\v} \left(\frac{\varphi\circ T_{\eps,\delta}}{\tau'\circ \pi_{\q+\v}}\right)&=\partial_{\q+\v} \left(\varphi\circ T_{\eps,\delta} \right)\frac 1 {\tau'\circ \pi_{\q+\v}}- \left(\varphi\circ T_{\eps,\delta} \right)\frac {\tau''\circ\pi_{\q+\v}} {(\tau')^2\circ \pi_{\q+\v}}.
\end{align*}
Moreover, 
\begin{align*}
\partial_{\q}(\varphi \circ T_{\eps,\delta})=\left\{
\begin{array}{l r c}
\partial_{\q} \varphi \circ T_{\eps,\delta} \cdot \tau'\circ \pi_\q \text{ if } \vx\notin \Delta_{\eps,\q,\v}\\
\partial_{\q+\v} \varphi \circ T_{\eps,\delta} \cdot\tau'\circ \pi_\q \text{ otherwise.}
\end{array}\right.
\end{align*}
Since $\Delta_{\eps,\q,\v}=\Delta_{\eps,\q+\v,-\v}$, the identity applied on the expansion above can be written as:
\begin{align*}
\partial_\q\varphi \circ T_{\eps,\delta}=\left\{
\begin{array}{l r c}
\partial_\q \left( \frac{\varphi \circ T_{\eps,\delta}} {\tau'\circ \pi_\q}\right)+\left(\varphi\circ T_{\eps,\delta} \right)\frac {\tau''} {\tau'^2}\circ \pi_\q \text{ if } \vx\notin \Delta_{\eps,\q,\v}\\
\partial_{\q+\v} \left( \frac{\varphi \circ T_{\eps,\delta}} {\tau'\circ \pi_{\q+\v}}\right)+\left(\varphi\circ T_{\eps,\delta} \right)\frac {\tau''} {\tau'^2}\circ\pi_{\q+\v} \text{ otherwise.}
\end{array}\right.
\end{align*}
Thus, we recursively get, for $n \in \N$, and $\v_1,\dots,\v_n \in V \cup \{0\}$ and any
$\vx\in \Delta_{\eps,n,\q,\v_1,\dots,\v_{n-1}}:=\bigcap_{k=0}^{n-1} T_{\eps,\delta}^{-k}\Delta_{\eps,\q+\sum_{i=1}^{k}\v_i,\v_{k+1}}$

\begin{align*}
\partial_\q\varphi \circ T_{\eps,\delta}^n=
\begin{array}{l r c}
\partial_{\q+\sum_{i=1}^n\v_i} \left( \frac{\varphi \circ T_{\eps,\delta}^n} {(\tau^{n})'\circ \pi_{\q+\sum_{i=1}^n\v_i}}\right)+\left(\varphi\circ T_{\eps,\delta}^n \right)\frac {(\tau^{n})''} {((\tau^{n})')^2}\circ \pi_{\q+\sum_{i=1}^n\v_i} 
\end{array}
\end{align*}
Notice first that for any two distinct tuples $(\v_1,\dots,\v_n)$ and $(\v_1',\dots,\v_n')$, $\Delta_{\eps,n,\q,\v_1,\dots,\v_n}\cap\Delta_{\eps,n,\q,\v_1',\dots,\v_n'}=\emptyset$
and that any element $\q'=\q+\sum_{i=1}^n\v_i$ lies in a $d$-dimensional box $G_\q^n$ around $\q$ with $|G_\q^n|=2dn+1$. Thus, denoting 
$$\tilde \Delta_{\eps,n,\q,\q',\v}:=\bigcup_{(\v_1,\dots,\v_n),\q+\sum \v_i=\q'}\Delta_{\eps,n,\q,\v_1,\dots,\v_n,\v},$$
we have
\begin{align}\label{eq:comp}
&\int  \partial_\q\varphi d{\LG_{\delta,n,a_i,0}'\mu}= \int \partial_\q\varphi \circ T^n_{\eps,\delta} 1_{ { B_{\delta,n,a_i}^0}}d\mu\\\notag
&= \sum_{\q' \in G_\q^n}\sum_{\v\in V\cup\{0\}}\alpha^{-n}\int_{\tilde\Delta_{\eps,n,\q,\q',\v}\cap {B_{\delta,n,a_i}^0}} \partial_{\q'} \left(\alpha^n \frac{\varphi \circ T_{\eps,\delta}^n} {(\tau^{n})'\circ \pi_{\q'}}\right)d\mu\\\notag
&+\int_{\tilde\Delta_{\eps,n,\q,\q',\v}\cap  { B_{\delta,n,a_i}^0}}\left(\varphi\circ T_{\eps.\delta}^n \right)\frac {(\tau^{n})''} {((\tau^{n})')^2}\circ \pi_{q'}d\mu,
\end{align}
where we recall that $\alpha>1$ and $|\tau'(x)|\ge \alpha>1$.
The second term of \eqref{eq:comp} can be easily bounded by a constant times the weak norm of $\mu$. We now deal with the first term of \eqref{eq:comp}.  Let $(I_k)_{k\in \N}$ be the partition of $I$ into intervals of $C^2$-regularity for $\tau^n$ such that $I_k$ and $I_{k+1}$ are consecutive intervals and let $K_i:=\bigcup_{k}I_{2k+i}$ for $i\in \{0,1\}$.  We have
\begin{equation*}
\begin{split}
    &\sum_{\q' \in G_\q^n}\sum_{\v\in V\cup\{0\}}\alpha^{-n}\int_{\tilde\Delta_{\eps,n,\q,\q',\v}\cap { B_{\delta,n,a_i}^0}} \partial_{\q'} \left(\alpha^n \frac{\varphi \circ T_{\eps,\delta}^n} {(\tau^{n})'\circ \pi_{\q'}}\right)d\mu\\
    &=\alpha^{-n}
    \sum_{\q' \in G_\q^n}\sum_{\v\in V\cup\{0\}}\sum_{i\in \{0,1\}}\int{\bf 1}_{\tilde\Delta_{\eps,n,\q,\q',\v}\cap { B_{\delta,n,a_i}^0}\cap C_{i}^\q} \partial_{\q'} \left(\alpha^n \frac{\varphi \circ T_{\eps,\delta}^n} {(\tau^{n})'\circ \pi_{\q'}}\right)d\mu,
\end{split}
\end{equation*}
 where   
 $$C_{i}^\q:=\{\vx\in I^{\ZZ^d}, x_\q \in K_i\}.$$
Let
$$\psi_{\q',n,i}:=\alpha^n\frac{\varphi \circ T_\eps^n} {(\tau^{n})'\circ \pi_{\q'}}1_{ \tilde\Delta_{\eps,n,\q,\q',\v_n}\cap { B_{\delta,n,a_i}^0}\cap C_{i}^{\q'} }.$$
Notice that $\psi_{\q',n, i}$ is continuous on its non-zero set since $\tau$ is piecewise $C^2$ and piecewise onto. To build a good test function out of  $\psi_{\q',n, i}$, let 
$$\tilde \psi_{\q', n, i}=\psi_{\q', n, i}-\psi^*_{\q',n, i}$$
where, for any fix $\vx_{\neq \q'}\in I^{\mathbb Z^d \backslash \{\q'\}}$, $\psi^*_{\q',n,i}$ is piecewise linear   
in $x_{\q'}$, with $\psi^*_{\q', n, i}=0$ on $\tilde\Delta_{\eps,n,\q,\q',\v_n}\cap  { B_{\delta,n,a_i}^0}\cap C_{i}^{\q'}$ and $\psi^*_{\q', n, i}(x)=\alpha^n\frac{\varphi \circ T_\eps^n} {(\tau^{n})'\circ \pi_{\q'}}(\vx)$ if $x\in \partial_{\q'}\tilde\Delta_{\eps,n,\q,\q',\v_n}\cap  { B_{\delta,n,a_i}^0}\cap C_{i}^{\q'}$ where $\partial_{\q'}\tilde\Delta_{\eps,n,\q,\q',\v_n}\cap { B_{\delta,n,a_i}^0}\cap C_{i}^{\q'}$ is the boundary of the set $\tilde\Delta_{\eps,n,\q,\q',\v_n}\cap { B_{\delta,n,a_i}^0}\cap C_{i}^{\q'}$ for the $x_{\q'}$ coordinates. Therefore, by its definition, the function $\tilde{\psi}_{\q',n,i}\in \mathcal D_1$ and $|\partial_{\q',n,i}\psi^*|\leq \frac{|\varphi|_\infty}{l_{\q',n^*,a_1,0,\delta}}$ where $l_{\q',n^*,a_1,0,\delta}$ is the minimal distance between two connected components of $C_{\q',0}$ and $C_{\q',1}$ where $C_{\q',i}:=\{x_{\q'},(x_{\q'},\vx_{\neq \q'})\in \tilde\Delta_{\eps,n,\q,\q',\v_n}\cap  { B_{\delta,n,a_i}^0}\cap C_{i}^{\q'}\}$. By construction, there is a finite number of connected subsets of $C_{\q',i}$, thus $l_{\q',n^*,a_1,0,\delta}>0$. Consequently,
\begin{align}\label{eq:LYF}
\int\partial_{\q'} \psi_{\q',n,i}d\mu &=\int \partial_{\q'}\tilde{\psi}_{\q',n,i}d\mu+\int \partial_{\q'}\psi^*_{\q',n,i}d\mu\\\notag
&\leq \|\mu\|+\frac{1}{l_{\q',n^*,a_1,0,\delta}}|\mu|.
\end{align}
Notice that when $g_\delta :=1$ and setting $a_1:=1$ then the set $B_{\delta,n,a_i}^0=I^{\mathbb Z^d}$ and the distance $l_{\q',n^*,a_1,0,\delta}$ does not depend on $\delta$. Thus we obtain from the equation above the uniform Lasota-Yorke iequality for $\LG_{\eps,\delta}$ : There is $C>0$ such that for any $n\in \N$ and $\mu \in \B$,
\begin{align}\label{eqlyvanil}
    \|\LG_{\eps,\delta}^n\mu\|\leq \sigma_0^n\|\mu\|+C|\mu|.
\end{align}

Now, choose $n^*\in \N$ large enough so that $2\alpha^{-n^*} (2d+1)|G_{\q'}^{n^*}|<\sigma_0<\frac 1 2$ and use \eqref{eq:LYF} in inequality \eqref{eq:comp} to obtain for any couple $(a_i,0)$ with $i\neq 1$,
\begin{equation*}
\begin{split}
&\int \partial_\q\varphi d{\LG_{\delta,n^*,a_i,0}'}\mu\\
&\leq\sigma_0\|\mu\|+2 (2d+1)|G_{\q'}^{n^*}|\left(\frac{1}{l_{\q', n^*,a_i,0,\delta}}+2\left|\frac {(\tau^{n^*})''} {((\tau^{n^*})')^2}\right|_\infty\right)|\mu|\\
&\le\sigma_0\|\mu\|+ C_{n^*}|\mu|,
\end{split}
\end{equation*}
where $C_{n^*}>0$ can be chosen independently of $\eps$ and $\delta$ , for $\eps>0$ and $\delta$ sufficiently small. Indeed, for such couple $(a_i,b)$, for any $\delta \leq \delta'$, $B_{\delta,n,a_i}^0\subset B_{\delta',n,a_i}^0$ are made of shrinking subintervals, and thus eventually $l_{\q',n^*,a_i,0,\delta}\geq l_{\q',n^*,a_i,0,\delta'}$. 
Similarly, since for any $\delta \leq \delta'$ and $B_{\delta,n,a_1}^1\subset B_{\delta',n,a_1}^1$, and thus $l_{\q',n^*,a_1,1,\delta}\geq l_{\q',n^*,a_1,1,\delta'}$ we also obtain the above Lasota-Yorke inequality for $(a_1,1)$,
\begin{equation*}
\begin{split}
&\int \partial_\q\varphi d{\LG_{\delta,n^*,a_1,1}'}\mu\\
&\leq\sigma_0\|\mu\|+2 (2d+1)|G_{\q'}^{n^*}|\left(\frac{1}{l_{\q',n^*,a_1,1,\delta}}+2\left|\frac {(\tau^{n^*})''} {((\tau^{n^*})')^2}\right|_\infty\right)|\mu|\\
&\le\sigma_0\|\mu\|+ C_{n^*}|\mu|,
\end{split}
\end{equation*}

We deduce from equation \eqref{eqlymauv} and  \eqref{eqlyvanil} that

\begin{align}\label{eqlyb1}
\|\LG_{\delta,n^*,a_1,0}^{n^*}\mu\|&\le \|\LG_{\eps,\vp^*}\mu-\LG_{\delta,n^*,a_1,1}\mu\|\nonumber\\
&\le 2\sigma_0\|\mu\|+ 2C_{n^*}|\mu|.
\end{align}

Now we choose $m$ such that $2(mn^*)^r\sigma_0^m<1$ and sum the $m$-iterates of inequality \eqref{eqlymauv} over every $a_i$ to get a Lasota-Yorke inequality uniform in $\delta$. There is some constant $K>0$ such that for any $\mu \in \B$, 
\begin{align*}
    \|\LG_{\eps,g_\delta}^{mn^*}\mu\|&\leq \sum_{i\leq m_{mn^*}}|a_i|\|\LG_{\delta,n^*,a_i,0}^{n^*}\mu\|\\
    &\leq 2(mn^*)^r\sigma_0^m \|\mu\|+ 2K(mn^*)^rC_{n^*}|\mu|.
\end{align*}
Thus, iterating the above inequality, for any $n\in \mathbb N$ and $\mu\in \mathcal B$,
$$\|\LG_{\eps,g_\delta}^{n}\mu\|\le  C \sigma^n\|\mu\|+ C|\mu|,$$
for some $\sigma\in(0,1)$ and $C>0$.
\end{proof}

\end{document}